\documentclass[review]{elsarticle}

\usepackage{lineno,hyperref}
\modulolinenumbers[5]
\usepackage{multirow}
\usepackage[english]{babel}
\usepackage{amsmath}
\usepackage{bm}
\usepackage{verbatim}
\usepackage{here}
\usepackage[matha,mathx]{mathabx}
\usepackage{mathrsfs}
\usepackage{calligra}
\usepackage{calrsfs}
\usepackage{amsfonts}

\usepackage{amssymb}
\usepackage{graphicx}

\usepackage{tikz}

\usepackage{algorithm}
\usepackage{algorithmic}
\usepackage{graphicx}
\usepackage{amsfonts}    

\usepackage{lipsum}
\newtheorem{definition}{Definition}
\newtheorem{proposition}{Proposition}
\newtheorem{theorem}{Theorem}
\newtheorem{proof}{Proof}
\newtheorem{Remark}{Remark}
\journal{JCAM}

\begin{document}

\begin{frontmatter}

\title{Einstien-Multidimensional Extrapolation methods}




\author[mymainaddress]{A. H. Bentbib}
\ead{a.bentbib@uca.ac.ma}
\author[mysecondaryaddress]{K. Jbilou}
\ead{khalide.jbilou@univ-littoral.fr}
\author[mymainaddress]{R. Tahiri\corref{mycorrespondingauthor},\fnref{myfootnote}}
\cortext[mycorrespondingauthor]{Corresponding author}
\fntext[myfootnote]{The third author contribution is major in this work}
\ead{ridwane.tahiri@ced.uca.ma}

\address[mymainaddress]{Faculty of Science and Technology Marrakech, University Cadi Ayyad, BP 549, 42 000 Marrakech, Morocco}
\address[mysecondaryaddress]{Universit\'e du Littoral Cote d'Opale, LMPA, 50 rue 
	F. Buisson, 62228 Calais-Cedex, France and University Mohammed VI, Benguerirr, Morocco}

\begin{abstract}
In this paper, we   present a new framework for   the recent  multidimensional extrapolation methods:  Tensor Global Minimal Polynomial (TG-MPE) and Tensor Global Reduced Rank Extrapolation (TG-RRE) methods. We develop  a  new  approach  to the   one presented  in \cite{17}. The proposed framework highlights, in addition their  polynomial feature, the connection of TG-MPE and TG-RRE with nonlinear Krylov subspace methods. A unified algorithm is proposed for their implemention. Theoretical results are given and some numerical experiments on linear and nonlinear problems are considered to confirm the performance of the proposed algorithms.
\end{abstract}

\begin{keyword}
Extrapolation methods \sep Multilinear\sep Tenso\sep, Einstein product\sep Krylov subspace method
\end{keyword}

\end{frontmatter}


\section{Introduction}\label{sec1}
Extrapolation methods \cite{1,20,21,22,23,24, 7, 10, 34} are useful tools for speeding up the convergence rate of sequences , they transform the basic iterates to a new sequence
that converges faster to the same limit of the initial sequence. Unfortunately,
there is no extrapolation method that can accelerate the convergence of all
sequences, and each method concern just a limited class of sequences, hence the
necessity of adapting or building new extrapolation methods that are suitable
for each type and class of sequences. For vector and matrix sequences, There
are two categories of extrapolation methods: Polynomial methods such that
the minimal polynomial extrapolation method (MPE) proposed by Cabay
and Jackson \cite{10}, the reduced rank extrapolation method (RRE) introduced
by Kaniel and Stein  \cite{25} and Mesina \cite{27}, and the modified minimal polynomial extrapolation method (MMPE) of Brezinski \cite{6}, Pugachev \cite{30}, Sidi,
Ford, and Smith \cite{35}. Epsilon algorithms like the vector epsilon algorithm
(VEA) of Wynn \cite{36} which is a vectorization of the scalar epsilon algorithm
(SEA) proposed by the same author in \cite{37} (which is a recursive procedure for
implementing the transformation of Shanks \cite{33}), and the topological epsilon
algorithm (TEA) of Brezinski \cite{5}. For accelerating multidimensional sequences,
that interest us in this work, Tensor Global Minimal Polynomial Extrapolation method (TG-MPE) and Tensor Global Reduced Rank Extrapolation
method (TG-RRE) are the first two Extrapolation methods presented in \cite{17}
for this purpose. They are introduced as projection methods that are, when
applied to linear iterative process, equivalent to the Krylov subspace methods Arnoldi and GMRES. As they are a generalizations of the well known
polynomial extrapolation methods MPE (Minimal Polynomial Extrapolation
) and RRE (Reduced Rank Extrapolation), they might inherit the polynomial
feature, in other words, they can be presented in the context of polynomial
methods. The aim of this paper is to provide a suitable framework that reveals
the polynomial type of these methods, as well as, justify their nonlinear Krylov conterpart when applied to nonlinear problems.\\ 
\noindent The next section is devoted for some preliminaries ans basic properties
about tensors. In  Section 3, starting from a class of linear iterative
sequences, and by the use of the generalized notion of minimal polynomial,
we expose the adopted polynomial approach to determine the limits of this kind  of
sequences. Section 4 provides, via least-squares problems, the definitions of
TG-MPE and TG-RRE while Section 5 is devoted to a unified algorithm to
implement them. In Section 6, we have highlighted the Krylov type counterpart of TG-MPE and TG-RRE, and explain how can be thought as nonlinear Krylov subspace methods when we applied them to nonlinear sequences.  In the last Section, we gpresent some numerical tests  that confirm the feasibility and effectiveness of the proposed approaches.
\section{Preliminaries and Notations}
In this section, we summarize some of the basic tools about tensors and their computations that will be used in the remainder of this paper.

\begin{definition}(\cite{4})
	
	A tensor is a multidimensional array  whose elements are referred by using multiple indices. The number of indices ( modes or ways )  is called the order of the tensor. For a given  $N$-order tensor $\mathcal{A} \in \mathbb{R}^{I_{1} \times I_{2} \times I_{3} \times \cdots \times I_{N}}$ we use the following  notation
	
	\begin{equation}\label{Eq1}
		\mathcal{A}=\left(\mathcal{A}_{i_{1} i_{2} \ldots i_{n-1} i_{n} i_{n+1} \ldots i_{N}}\right)_{1 \leq i_{n} \leq I_{n} ; 1 \leq n \leq N},
	\end{equation}
	where $\mathcal{A}_{i_{1} i_{2} \ldots i_{n-1} i_{n} i_{n+1} \ldots i_{N}}$ are the entries of $\mathcal{A}$. 
\end{definition}

\noindent For a square tensor  $\mathcal{A} \in \mathbb{R}^{I_{1} \times \ldots \times I_{N} \times I_{1} \times \ldots \times I_{N}}$, the trace of $\mathcal{A}$ is the scalar given by

\begin{equation}\label{Eq2}
	\operatorname{tr}(\mathcal{A})=\sum_{i_{1} \ldots i_{N}} \mathcal{A}_{i_{1}, \ldots i_{N} i_{1} \ldots i_{N}}.
\end{equation}
For a tensor $\mathcal{A} \in \mathbb{R}^{I_{1} \times \ldots \times I_{N} \times J_{1} \times \cdots \times J_{M}}$, the transpose of $\mathcal{A}$ is the tensor $\mathcal{A}^{T}$ of size $J_{1} \times \cdots \times J_{M} \times I_{1} \times \ldots \times I_{N}$ whose elements are given by 

\begin{equation}\label{Eq3}
	\left(\mathcal{A}^{T}\right)_{j_{1} \ldots j_{M} i_{1}, \ldots i_{N}}=\mathcal{A}_{i_{1}, \ldots i_{N} j_{1} \ldots j_{M}}.
\end{equation}

\begin{definition} (\cite{4}). The Einstein product of two tensors $\mathcal{A} \in \mathbb{R}^{I_{1} \times \ldots \times I_{N} \times J_{1} \times \cdots \times J_{M}}$ and $\mathcal{B} \in \mathbb{R}^{J_{1} \times \cdots \times J_{M} \times K_{1} \times \cdots \times K_{L}}$, is the $I_{1} \times \ldots \times I_{N} \times K_{1} \times \cdots \times K_{L}$ tensor denoted by $\mathcal{A} *_{M} \mathcal{B}$ whose entries are given by
	
	\begin{equation}\label{Eq4}
		\left(\mathcal{A} *_{M} \mathcal{B}\right)_{i_{1}, \ldots i_{N} k_{1}, \ldots k_{M}}=\sum_{j_{1} \ldots j_{M}} \mathcal{A}_{i_{1}, \ldots i_{N} j_{1} \ldots j_{M}} \mathcal{B}_{j_{1} \ldots j_{M} k_{1}, \ldots k_{L}}.
	\end{equation}
\end{definition} 
\begin{definition} (\cite{4}) Let $\mathcal{A} \in \mathbb{R}^{J_{1} \times \cdots \times J_{M} \times J_{1} \times \cdots \times J_{M}}$ and let  
	$\mathcal{I}=\left[\delta_{i_{1} \ldots i_{N} j_{1} \ldots j_{N}}\right]$ denotes the identity tensor whose elements are as	
	\begin{equation}\label{Eq6}
		\delta_{i_{1} \ldots i_{N} j_{1} \ldots j_{N}}=\prod_{k=1}^{N} \delta_{i_{k} j_{k}} \text { with } \delta_{i_{k} j_{k}}=1 \text { if } i_{k}=j_{k} \text { and } 0 \text { else. }
	\end{equation}
	If there exists  a tensor $\mathcal{B} \in$ $\mathbb{R}^{J_{1} \times \cdots \times J_{M} \times J_{1} \times \cdots \times J_{M}}$ such that
	\begin{equation}\label{Eq5}
		\mathcal{A} *_M \mathcal{B}=\mathcal{B} *_M \mathcal{A}=\mathcal{I}, 
	\end{equation}
	then $\mathcal{A}$ is said to be invertible, and $\mathcal{B}$ is called the inverse of $\mathcal{A}$, denoted as $\mathcal{A}^{-1}$. 
\end{definition} 

\begin{definition}  (Inner product of two tensors (\cite{14})). Let $\mathcal{A}$ and $\mathcal{B}$ two tensors of the same size $I_{1} \times I_{2} \times \cdots \times I_{N} \times I_{1} \times I_{2} \times \cdots \times I_{N}$, the (Frobenious) inner product of $\mathcal{A}$ and $\mathcal{B}$ is the scalar defined as
	\begin{equation}\label{Eq7}
		\langle\mathcal{A}, \mathcal{B}\rangle=\operatorname{tr}\left(\mathcal{A} *_{N} \mathcal{B}\right)=\sum_{i_{N}=1}^{I_{N}} \cdots \sum_{i_{1}=1}^{I_{1}} \mathcal{A}_{i_{1} \ldots i_{N}} \mathcal{B}_{i_{1} \ldots i_{N}}
	\end{equation}
	leading to the tensor norm
	\begin{equation}\label{Eq8}
		\|\mathcal{A}\|_{F}=\sqrt{\langle\mathcal{A}, \mathcal{A}\rangle}.
	\end{equation}
\end{definition}
\begin{definition}  \cite{26} Given a tensor $\mathcal{A} \in \mathbb{R}^{I_{1} \times I_{2} \times I_{3} \times \cdots \times I_{N}}$ and a matrix $M \in \mathbb{R}^{J_{n} \times I_{n}}$. Then the $n$-mode product of the tensor $\mathcal{A}$ and $M$, denoted by $\mathcal{A} \times{ }_{n} M$, is the tensor of size  $\left(I_{1} \times I_{2} \ldots \times I_{n-1} \times J_{n} \times\right.$ $\left.I_{n+1} \times \cdots \times I_{N}\right)$ defined by
	
	\begin{equation}\label{Eq9}
		\left(\mathcal{A} \times_{n} M\right)_{i_{1} i_{2} \ldots i_{n-1} j_{n} i_{n+1} \ldots i_{N}}=\sum_{i_{n}} \mathcal{A}_{i_{1} i_{2} \ldots i_{n-1} i_{n} i_{n+1} \ldots i_{N}} M_{j_{n} i_{n}}. 
	\end{equation}
	The $n$-mode product of the tensor $\mathcal{A}$ and a vector $w \in \mathbb{R}^{I_{n}}$, denoted by $\mathcal{A} \bar{\times}_{n} w$ is the subtensor of order $\left(I_{1} \times I_{2} \ldots \times I_{n-1} \times I_{n+1} \times \cdots \times I_{N}\right)$ defined as
	
	\begin{equation}\label{Eq10}
		\left(\mathcal{A} \bar{\times}_{n} w\right)_{i_{1} i_{2} \ldots i_{n-1} i_{n+1} \ldots i_{N}}=\sum_{i_{n}=1}^{I_{n}} \mathcal{A}_{i_{1} i_{2} \ldots i_{n-1} i_{n} i_{n+1} \ldots i_{N}} w_{i_{n}}.
	\end{equation}
\end{definition} 

\begin{definition}  (\cite{4})Let $I=I_{1} I_{2} \ldots I_{N}, J=J_{1} J_{2} \ldots J_{M}$ 
	and define  the function from the space of tensors onto the space of matrices
	\begin{equation}\label{Eq11}
		\begin{aligned}
			\phi_{I J}: \mathbb{R}^{I_{1} \times \ldots \times I_{N} \times J_{1} \times \cdots \times J_{M}} & \longrightarrow \mathbb{R}^{I \times J}, \\
			\mathcal{A} & \longmapsto \phi_{I J}(\mathcal{A})=A, 
		\end{aligned}
	\end{equation}
	such that the components of the matrix $A$ are given by
	
	\begin{equation}\label{Eq12}
		A_{p q}=\mathcal{A}_{i_{1}, \ldots i_{N} j_{1} \ldots j_{M}}, 
	\end{equation}
	with $$p=i_{N}+\sum_{k=1}^{N-1}\left(\left(i_{k}-1\right) \prod_{m=k+1}^{N} I_{m}\right)$$ and $$q=j_{M}+\sum_{k=1}^{N-1}\left(\left(j_{k}-1\right) \prod_{m=k+1}^{M} J_{m}\right).$$
\end{definition} 
\noindent In the case where $\mathcal{A} \in \mathbb{R}^{I_{1} \times \ldots \times I_{N}},$ its image $\phi_{I J}(\mathcal{A})=\phi_{I 1}(\mathcal{A})$ is a vector of $\mathbb{R}^{I} \quad(\mathrm{~J}=1)$.
\begin{proposition}(\cite{14})
Let  $\mathbf{M}_{I I}$ be the group of all invertible $I \times I$ matrices, and let

\begin{equation}\label{Eq13}
	\mathbf{T}=\left\{\mathcal{T} \in \mathbb{R}^{I_{1} \times \ldots \times I_{N} \times I_{1} \times \ldots \times I_{N}}\; \text{such that}\;  \phi_{I I}(\mathcal{T})\;  \text{is nonsingular} \right\}.
\end{equation}
\noindent Then, the restriction of the function ${\phi_{I I}}^T$  on $\mathbf{T}$ :
$\left(\mathbf{T}, *_{N}\right) \longrightarrow(\mathbf{M_{II}}, .)$
is a group isomorphisme. 
\end{proposition} 
\begin{proposition}(\cite{14})
	Let $\mathcal{A} \in \mathbb{R}^{I_{1} \times \cdots \times I_{N} \times J_{1} \times \cdots \times J_{M}}$, $\mathcal{B} \in \mathbb{R}^{J_{1} \times \cdots \times J_{M} \times K_{1} \times \cdots \times K_{L}}$, and $\mathcal{C} \in \mathbb{R}^{I_{1} \times \ldots \times I_{N} \times K_{1} \times \cdots \times K_{L}}$. Then 
	\begin{equation}\label{Eq15}
		\mathcal{A} *_{M} \mathcal{B}=\mathcal{C} \Longleftrightarrow \phi_{I J}(\mathcal{A}) \phi_{J K}(\mathcal{B})=\phi_{I K}(\mathcal{C}).
	\end{equation}
\end{proposition} 

\begin{definition} (\cite{14}) Let $\mathcal{A} \in \mathbb{R}^{I_{1} \times \ldots \times I_{N} \times I_{1} \times \ldots \times I_{N}}$ be a square tensor. If there exists a nonzero tensor $\mathcal{X} \in \mathbb{R}^{I_{1} \times \ldots \times I_{N}}$ and a scalar $\lambda$ such that
	\begin{equation}\label{Eq16}
		\mathcal{A} *_{N} \mathcal{X}=\lambda \mathcal{X}, 
	\end{equation}
	then $\lambda$ is called an eigenvalue of $\mathcal{A}, \mathcal{X}$ is called an eigentensor of $\mathcal{A}$ corresponding to $\lambda$. The set of all eigenvalues of $\mathcal{A}$ is denoted as $\sigma(\mathcal{A})$.
\end{definition}
\noindent In view of (\ref{Eq15}), we remark that $\lambda$ is an eigenvalue of the matrix $\phi_{I I}(\mathcal{A})$ associated to the eigenvector $\phi_{I 1}(\mathcal{X}),$   wich gives 
\begin{equation}\label{Eq17}
	\sigma(\mathcal{A})=\sigma_{c}\left(\phi_{I I}(\mathcal{A})\right),
\end{equation}
with $\sigma_{c}\left(M\right)$ stands to the classical spectrum of a square matrix $M$.
The  spectral radius  of the square tensor $\mathcal{A}$ is denoted by 
\begin{equation}\label{Eq18}
	\rho(\mathcal{A})=\max_{\lambda \in \sigma(\mathcal{A})}\mid\lambda\mid.
\end{equation}
\noindent Similarly to the matrix case, we can define the characteristic polynomial of a  square tensor.
\begin{definition} Let $\mathcal{A} \in \mathbb{R}^{I_{1} \times \ldots \times I_{N} \times I_{1} \times \ldots \times I_{N}}$ be a given  square tensor. We define the characteristic polynomial of $\mathcal{A}$ by 
	\begin{equation}\label{Eq19}
		{\mathcal P}_{\mathcal A}(\lambda)=\prod_{\lambda_{i} \in \sigma(\mathcal{A})}\left(\lambda_i-\lambda \right).
	\end{equation}
\end{definition}
\noindent Denote $\mathcal{A}^{k}=\mathcal{A} *_{N} \mathcal{A}^{k-1}$ with $\mathcal{A}^{0}=\mathcal{I}$. We state the following result that extend to tensor form the the classical Cayley-Hamilton theorem. Recall that $I=I_{1} I_{2} \ldots I_{N}$.

\begin{theorem}
	Let ${\mathcal P}_{\mathcal A}(\lambda)=\sum_{k=0}^{I} c_{k} \lambda^{k}$ the characteristic polynomial of $\mathcal{A} \in$ $\mathbb{R}^{I_{1} \times \ldots \times I_{N} \times I_{1} \times \ldots \times I_{N}}$ then,
	
	\begin{equation}\label{Eq20}
		{\mathcal P}_{\mathcal A}(\mathcal{A})=\sum_{k=0}^{I} c_{k} \mathcal{A}^{k}=\mathcal{O}.
	\end{equation}
\end{theorem}
\begin{proof} 
	Using the fact that  $\phi_{I I}\left(\mathcal{A}^{k}\right)=\left[\phi_{I I}(\mathcal{A})\right]^{k}$, we get the following relations\\
	$\phi_{I I}(\mathcal{P}_\mathcal{A}(\mathcal{A}))=\phi_{I I}\left(\sum_{k=0}^{I} c_{k} \mathcal{A}^{k}\right)=\sum_{k=0}^{I} c_{k} \phi_{I I}\left(\mathcal{A}^{k}\right)=\mathcal{P}_\mathcal{A}\left(\phi_{I I}(\mathcal{A})\right)=0$, and since $\phi_{I I}$ is injective, it follows  ${\mathcal P}_{\mathcal A}(\mathcal{A})=\mathcal{O}$. 
\end{proof} 
\medskip 

\begin{definition} \label{def10}  
	Let $\mathcal{A} \in \mathbb{R}^{I_{1} \times \ldots \times I_{N} \times I_{1} \times \ldots \times I_{N}}$ and a nonzero tensor $\mathcal{X} \in \mathbb{R}^{I_{1} \times \ldots \times I_{N}}$. The minimal polynomial of $\mathcal{A}$ (respectively with respect to $\mathcal{X}$ ) is the polynomial, denoted by $\bar{\mathcal{P}}_\mathcal{A}$ (respectively  $\bar{\mathcal{P}}_\mathcal{A}^{[\mathcal{X}]})$, with smallest degree such that $\bar{\mathcal{P}}(\mathcal{A})=\mathcal{O}$ (respectively  $\bar{\mathcal{P}}_\mathcal{A}^{[\mathcal{X}]}(\mathcal{A}) *_{N}$ $\mathcal{X}=\mathcal{O}).$
\end{definition}
\medskip 
\begin{theorem}
	The minimal polynomial $\bar{\mathcal{P}}_\mathcal{A}^{[\mathcal{X}]}$ of the tensor $\mathcal{A}$ with respect to $\mathcal{X}$ exists and is unique.
\end{theorem}
\medskip
\begin{proof}
	The proof is similar to the matrix case.
\end{proof} 

\section{Determination of limit via minimal polynomial}
In this section we will be interested in a class of tensor sequences generated by the (multi)linear process given by (\ref{Eq21}). 
We will demonstrate the utilization of the tensor minimal polynomial to determine the limit denoted as $\bar{\mathcal{X}}$. This limit can be represented, as we will explore, through a finite number of terms within the underlying sequence. This section serves as an introduction to the construction of TG-MPE and TG-RRE, which will be discussed in the following section, emphasizing their polynomial features.

\noindent Let $\left\{\mathcal{X}_{n}\right\}$ be the sequence of tensors in $\mathbb{R}^{I_{1} \times I_{2} \times I_{3} \times \cdots \times I_{N}}$ defined by the linear process

\begin{equation}\label{Eq21}
	\mathcal{X}_{n+1}=\mathcal{M} *_{N} \mathcal{X}_{n}+\mathcal{B}, 
\end{equation}
where  $\mathcal{M} \in \mathbb{R}^{I_{1} \times \cdots \times I_{N} \times I_{1} \times \cdots \times I_{N}}$ and $\mathcal{B} \in \mathbb{R}^{I_{1} \cdots \times I_{N}}$ are given tensors.
Then, if $\left\{\mathcal{X}_{n}\right\}$ converges to the  limit $\bar{\mathcal{X}}$,  we get 

\begin{equation}\label{Eq22}
	(\mathcal{I}-\mathcal{M}) *_{N} \bar{\mathcal{X}}=\mathcal{B},
\end{equation}
that is, the tensor $(\mathcal{I}-\mathcal{M})$ is invertible and the convergence of $\left\{\mathcal{X}_{n}\right\}$ implies that $\rho(\mathcal{M})<1$.

\noindent Let us define the forward difference
\begin{equation}\label{Eq23}
	\mathcal{D}_{n}=\mathcal{X}_{n+1}-\mathcal{X}_{n}, \quad n=0,1, \ldots,
\end{equation}
and the error tensor 
\begin{equation}\label{Eq24}
	\mathcal{E}_{n}=\mathcal{X}_{n}-\bar{\mathcal{X}}, \quad n=0,1, \ldots
\end{equation}
Using (\ref{Eq23}) and (\ref{Eq24}), it follows that

\begin{equation}\label{Eq25}
	\mathcal{D}_{n+1}=\mathcal{M} *_{N} \mathcal{D}_{n}, \quad \mathcal{E}_{n+1}=\mathcal{M} *_{N} \mathcal{E}_{n}, \quad n=0,1, \ldots
\end{equation}
Therefore,
\begin{equation}\label{Eq26}
	\mathcal{D}_{n}=(\mathcal{M})^{n} *_{N} \mathcal{D}_{0}, \quad \mathcal{E}_{n}=(\mathcal{M})^{n} *_{N} \mathcal{E}_{0}, \quad n=0,1, \ldots
\end{equation}
We can relate $\mathcal{E}_{n}$ to $\mathcal{D}_{n}$ via
\begin{equation}\label{Eq27}
	\mathcal{D}_{n}=(\mathcal{I}-\mathcal{M}) *_{N} \mathcal{E}_{n} \quad \text { and } \quad \mathcal{E}_{n}=(\mathcal{I}-\mathcal{M})^{-1} *_{N} \mathcal{D}_{n},
\end{equation}
We will also need the following simple relations
\begin{equation}\label{Eq28}
	\mathcal{D}_{n+i}=(\mathcal{M})^{i} *_{N} \mathcal{D}_{n}, \quad \mathcal{E}_{n+i}=(\mathcal{M})^{i} *_{N} \mathcal{E}_{n}, \quad n=0,1, \ldots
\end{equation}

\begin{theorem}\label{Th5}
	The minimal polynomials $\bar{\mathcal{P}}_{\mathcal{M}}^{\left[\mathcal{E}_{n}\right]}$ and $\bar{\mathcal{P}}_{\mathcal{M}}^{\left[\mathcal{D}_{n}\right]}$ of $\mathcal{M}$ with respect to $\mathcal{E}_{n}$ and $\mathcal{D}_{n}$ are the same.
\end{theorem} 
\begin{proof}
	Following definition \ref{def10}, one has
	\begin{equation}\label{Eq29}
		\bar{\mathcal{P}}_{\mathcal{M}}^{\left[\mathcal{E}_{n}\right]}(\mathcal{M}) *_{N} \mathcal{E}_{n}=\mathcal{O},
	\end{equation}
	and
	\begin{equation}\label{Eq30}
		\bar{\mathcal{P}}_{\mathcal{M}}^{\left[\mathcal{D}_{n}\right]}(\mathcal{M}) *_{N} \mathcal{D}_{n}=\mathcal{O}.
	\end{equation}
	Multiplying the equality (\ref{Eq29}) by $(\mathcal{I}-\mathcal{M})$, and recalling from (\ref{Eq27}) that $\mathcal{D}_{n}=(\mathcal{I}-\mathcal{M})*_{N} \mathcal{E}_{n}$, we obtain 
	$\bar{\mathcal{P}}_{\mathcal{M}}^{\left[\mathcal{E}_{k}\right]}(\mathcal{M}) *_{N} \mathcal{D}_{n}=\mathcal{O}$, this implies that $\bar{\mathcal{P}}_{\mathcal{M}}^{\left[\mathcal{D}_{k}\right]}$ divides $\bar{\mathcal{P}}_{\mathcal{M}}^{\left[\mathcal{E}_{k}\right]}$. On the other hand, using (\ref{Eq27}), we can rewrite (\ref{Eq30}) as $\bar{\mathcal{P}}_{\mathcal{M}}^{\left[\mathcal{D}_{k}\right]}(\mathcal{M}) *_{N}(\mathcal{I}-\mathcal{M}) *_{N} \mathcal{E}_{n}=\mathcal{O}$, which, upon multiplying by $(\mathcal{I}-\mathcal{M})^{-1}$, gives $\bar{\mathcal{P}}_{\left[\mathcal{D}_{k}\right]}(\mathcal{M}) *_{N} \mathcal{E}_{n}=\mathcal{O}$, this implies that $\bar{\mathcal{P}}_{\left[\mathcal{E}_{k}\right]}$ divides $\bar{\mathcal{P}}_{\left[\mathcal{D}_{k}\right]}.$ Therefore, $\bar{\mathcal{P}}_{\left[\mathcal{E}_{k}\right]} \equiv \bar{\mathcal{P}}_{\left[\mathcal{D}_{k}\right]}$.
\end{proof} 

\medskip
\begin{Remark}
	\noindent In practice, the limit $\bar{\mathcal{X}}$ is not known and then the errors $\mathcal{E}_{n}$, consequently $\bar{\mathcal{P}}_\mathcal{M}^{\left[\mathcal{E}_{k}\right]}$ can not be computed directly using the error $\mathcal{E}_{n}$. Theorem \ref{Th5} states that this  can be done using only our knowledge of the available tensors differences $\mathcal{D}_{n}$.
\end{Remark}
\medskip
\noindent Taking advantage of Theorem \ref{Th5} and   our knowledge of the sequence of differences $\mathcal{D}_{n}$, the following result show that the limit $\bar{\mathcal{X}}$ can be exactly determined via using just a finite number of terms of the sequence $\left\{\mathcal{X}_{k}\right\}$.

\medskip
\begin{theorem} Let $\bar{\mathcal{P}}_\mathcal{M}^{\left[\mathcal{D}_{n}\right]}$ be the minimal polynomial of $\mathcal{M}$ with respect to $\mathcal{D}_{n}$, and denote $d$ its degree
	\begin{equation}\label{Eq31}
		\bar{\mathcal{P}}_\mathcal{M}^{\left[\mathcal{D}_{n}\right]}(x)=\sum_{i=0}^{d} \theta_{i}^{(n)} x^{i}, \quad \theta_{d}^{(n)}=1.
	\end{equation}
	Then $\sum_{i=0}^{d} \theta_{i}^{(n)} \neq 0$, and the limit  $\bar{\mathcal{X}}$ can be expressed as
	\begin{equation}\label{Eq32}
		\bar{\mathcal{X}}=\frac{\sum_{i=0}^{d} \theta_{i}^{(n)} \mathcal{X}_{k+i}}{\sum_{i=0}^{d} \theta_{i}^{(n)}}.
	\end{equation}
\end{theorem}

\begin{proof} 
	Since $\bar{\mathcal{P}}_\mathcal{M}^{\left[\mathcal{D}_{n}\right]}$ is also the minimal polynomial of $\mathcal{M}$ with respect to $\mathcal{E}_{n}$, using  (\ref{Eq28}), we get 
	\begin{equation}\label{Eq33}
		\mathcal{O}=\bar{\mathcal{P}}_\mathcal{M}^{\left[\mathcal{D}_{n}\right]}(\mathcal{M}) *_{N} \mathcal{E}_{n}=\sum_{i=0}^{d} \theta_{i}^{(n)}(\mathcal{M})^{i} *_{N} \mathcal{E}_{n}=\sum_{i=0}^{d} \theta_{i}^{(n)} \mathcal{E}_{n+i}.
	\end{equation}
	Therefore
	\begin{equation}\label{Eq34}
		\sum_{i=0}^{d} \theta_{i}^{(k)} \mathcal{E}_{n+i}=\sum_{i=0}^{d} \theta_{i}^{(n)} \mathcal{X}_{n+i}-\left(\sum_{i=0}^{d} \theta_{i}^{(n)}\right) \mathcal{S} =\mathcal{O}.
	\end{equation}
	Since $\rho(\mathcal{M})<1$, one is not an eigenvalue of $\mathcal{M}$ and  then  $\sum_{i=0}^{d} \theta_{i}^{(n)}=\bar{\mathcal{P}}_{\left[\mathcal{D}_{n}\right]}(1) \neq 0$. Dividing by $\sum_{i=0}^{d} \theta_{i}^{(n)}$, we get the relation (\ref{Eq32}).
	Now,  setting 
	\begin{equation}\label{Eq35}
		\delta_{i}^{(n)}=\frac{\theta_{i}^{(n)}}{\sum_{j=0}^{d} \theta_{j}^{(n)}}, \quad i=1,2, \ldots, d, 
	\end{equation}
	we get 
	\begin{equation}\label{Eq36}
		\bar{\mathcal{X}}=\sum_{i=0}^{d} \delta_{i}^{(n)} \mathcal{X}_{n+i}, 
	\end{equation}
	and by construction
	\begin{equation}\label{Eq37}
		\sum_{i=0}^{d} \delta_{i}^{(n)}=1.
	\end{equation}
\end{proof}

Notice  that the scalars $\theta_{i}^{(n)},i=0, \ldots, d$ verify the equation
\begin{equation}\label{Eq38}
	\sum_{i=0}^{d} \theta_{i}^{(n)} \mathcal{D}_{n+i}=\mathcal{O}.
\end{equation}
Recalling $\theta_{d}^{(n)}=1$, we have 
\begin{equation}\label{Eq39}
	\sum_{i=0}^{d-1} \theta_{i}^{(k)} \mathcal{D}_{k+i}=-\mathcal{D}_{k+m}.
\end{equation}
Using $n$-mode product, (\ref{Eq39}) can be expressed as
\begin{equation}\label{Eq40}
	\mathbb{D}_{d-1}^{n} \bar{\times}_{(N+1)} \tilde{\theta}^{(n)}=-\mathcal{D}_{n+d}, \quad \tilde{\theta}^{(n)}=\left(\theta_{1}^{(n)}, \theta_{2}^{(n)}, \ldots, \theta_{d-1}^{(n)}\right)^{T},
\end{equation}
where $\mathbb{D}_{d}^{n}=\left[\mathcal{D}_{n}, \mathcal{D}_{n+1}, \ldots, \mathcal{D}_{n+d}\right]$ is the $I_{1} \times I_{2} \times I_{3} \times \cdots \times I_{N} \times d$ tensor such that the $j^{t h}$ frontal slice ( obtained by fixing the last index at $j$ ) $\left(\mathbb{D}_{d}^{n}\right)_{:,:, \ldots,:, j}=\mathcal{D}_{n+j}, 0 \leq j \leq d$.

\noindent Following Theorem \ref{Th5}, the equation (\ref{Eq40}) is consistent and has a unique solution $\tilde{\theta}^{(k)}$. Actually, that comes back to the uniqueness of the minimal polynomial. 

\noindent Invoking (\ref{Eq35}) and dividing (\ref{Eq38}) by $\sum_{i=0}^{d} \theta_{i}^{(n)}$, we find that the $\delta_{i}^{(n)}$ satisfy the constrained system
\begin{equation}\label{Eq41}
	\mathbb{D}_{d}^{n} \bar{\times}_{(N+1)} \delta^{(n)}=0, \quad \text { and } \quad \sum_{i=0}^{d} \delta_{i}^{(n)}=1, \quad \delta^{(n)}=\left(\delta_{1}^{(n)}, \ldots, \delta_{d}^{(n)}\right)^{T}.
\end{equation}

\noindent This is a consistent system of $(d+1)$ unknowns $\delta_{0}^{(n)}, \delta_{1}^{(n)}, \ldots, \delta_{d}^{(n)}$ that has a unique solution.

\noindent There are then two scenarios to determine the scalars $\delta_{0}^{(n)}, \delta_{1}^{(n)}, \ldots, \delta_{d}^{(n)}$ :\begin{enumerate}
	\item Indirectly: via the solution of the equation (\ref{Eq38}) and then set
	
	\noindent  $\delta_{i}^{(n)}=$ $\frac{\theta_{i}^{(n)}}{\sum_{i=0}^{d} \theta_{i}^{(n)}}, \quad i=0,1, \ldots, d.$
	
	\item Directly : via the solution of the constrained problem (\ref{Eq41}).
\end{enumerate}
As will be shown, each one of these two scenarios leads to an extrapolation method that can be used as an accelerator for speeding up slowly convergent sequences.

\section{Approximation via least squares problems}

The degree $d$ of the minimal polynomial can be very large and then  the solution of equations (\ref{Eq40}) and (\ref{Eq41}) could  be very expensive for the  computation time and storage requirements. Therefore, we have replaced the degree $d$  by a smaller integer $\bar{d}<<d$ to get approximate solutions. Given the minimality of degree $d$ and linear independence of the set $S_{\bar{d}}=\left\{\mathcal{D}_{n}, \mathcal{D}_{n+1}, \ldots, \mathcal{D}_{n+\bar{d}}\right\}$, the equations (\ref{Eq40}) and (\ref{Eq41}) are no longer consistent and have no  solutions in the ordinary sense. An alternative approach to overcame this obstacle is going throughout  least-squares solutions , indeed, such the former always exists and leads to an effective approximations to  limit $\bar{\mathcal{X}}$.

Taking into account this novel approach and what we have seen in the previous section, we have two available tracks (the following scenarios $1$ or $2$) to approximate the limit $\bar{\mathcal{X}}$.

1-  Solve for $\tilde{\theta}^{(n)}=\left(\theta_{0}^{(n)}, \theta_{1}^{(n)}, \ldots, \theta_{\bar{d}-1}^{(n)}\right)^{T}$ the system $(\ref{Eq40})$ in the least squares sense, this leads to the problem 
\begin{equation}\label{Eq42}
	\min _{\tilde{\theta}^{(n)}}\left\|\mathbb{D}_{\bar{d}-1}^{n} \bar{\times}_{(N+1)} \tilde{\theta}^{(n)}+\mathcal{D}_{n+\bar{d}}\right\|_{F}.
\end{equation}
\noindent We take $\theta_{\bar{d}}^{(n)}=1$ and  compute
\begin{equation}\label{Eq43}
	\delta_{i}^{(n)}=\frac{\theta_{i}^{(n)}}{\sum_{i=0}^{\bar{d}} \theta_{i}^{(n)}}, \quad i=0,1, \ldots, \bar{d}.
\end{equation}
Then, we set $\mathcal{T}_{n}^{(\bar{d})}=\sum_{i=0}^{\bar{d}} \delta_{i}^{(n)} \mathcal{X}_{n+i}$.

2- Solve for $\delta^{(n)}=\left(\delta_{0}^{(n)}, \delta_{1}^{(n)}, \ldots, \delta_{\bar{d}}^{(n)}\right)$ the constrained problem $(\ref{Eq41})$ in the least squares sense, this leads to solve the minimisation problem 
\begin{equation}\label{Eq44}
	\min _{\delta(n)}\left\|\mathbb{D}_{\bar{d}}^{n} \bar{\times}_{(N+1)} \delta^{(n)}\right\|_{F} \quad \text { subject~to } \quad  \sum_{i=0}^{\bar{d}} \delta_{i}^{(n)}=1,
\end{equation}
set $\mathcal{T}_{n}^{(\bar{d})}=\sum_{i=0}^{\bar{d}} \delta_{i}^{(n)} \mathcal{X}_{n+i}$ as an approximation to limit $\bar{\mathcal{X}}.$

For the purpose of making the implementation of problem (\ref{Eq44}) easier, we replace it with an equivalent unconstrained problem. We emphasize this through the following proposition.
\medskip
\begin{proposition}
	Let $\delta^{(n)}=\left(\delta_{0}^{(n)}, \delta_{1}^{(n)}, \ldots, \delta_{\bar{d}}^{(n)}\right)$ be the solution of (\ref{Eq44}). Then  we have
	\begin{equation}\label{Eq45}
		\mathcal{T}_{n}^{(\bar{d})}=\sum_{i=0}^{\bar{d}} \delta_{i}^{(n)} \mathcal{X}_{n+i}=\mathcal{X}_{n}+\sum_{j=0}^{\bar{d}-1} \mu_{j}^{(n)} \mathcal{D}_{n+j},
	\end{equation}
	
	\noindent where  $\mu^{(n)}=\left(\mu_{0}^{(n)}, \mu_{1}^{(n)}, \ldots, \mu_{\bar{d}}^{(n)}\right)^{T}$ is such that
	
	\begin{equation}\label{Eq46}
		\mu^{(n)}=\underset{x}{\operatorname{argmin}}\left\|\mathbb{W}_{\bar{d}-1}^{n} \bar{\times}_{(N+1)} x+\mathcal{D}_{n}\right\|_{F},
	\end{equation}
	
	\noindent and  $\mathbb{W}_{\bar{d}-1}^{n}=\left[\mathcal{W}_{n}, \mathcal{W}_{n+1}, \ldots, \mathcal{W}_{n+\bar{d}-1}\right]$ and $\mathcal{W}_{j}=\mathcal{D}_{j+1}-\mathcal{D}_{j}, \quad j=n, \ldots, n+$ $\bar{d}-1$.
\end{proposition} 

\medskip
\begin{proof}
	Using the notations $\mathcal{D}_{j}=\mathcal{X}_{j+1}-\mathcal{X}_{j}$ and $\mathcal{W}_{j}=\mathcal{D}_{j+1}-\mathcal{D}_{j}$, we have
	\begin{equation}\label{Eq47}
		\mathcal{X}_{n+i}=\mathcal{X}_{n}+\sum_{j=0}^{i-1} \mathcal{D}_{n+j} \quad \text { and } \quad \mathcal{D}_{n+i}=\mathcal{D}_{n}+\sum_{j=0}^{i-1} \mathcal{W}_{j}.
	\end{equation}
	Setting 
	\begin{equation}\label{Eq48}
		\mu_{j}^{(n)}=\sum_{i=j+1}^{\bar{d}} \delta_{i}^{(n)}=1-\sum_{i=0}^{j} \delta_{i}^{(n)}, \quad j=0,1 \ldots, \bar{d}-1,
	\end{equation}
	and recalling that $\sum_{i=0}^{\bar{d}} \delta_{i}^{(n)}=1$, we get 
	
	\begin{equation}\label{Eq49}
		\sum_{i=0}^{\bar{d}} \delta_{i}^{(n)} \mathcal{X}_{n+i}=\mathcal{X}_{n}+\sum_{j=0}^{\bar{d}-1} \mu_{j}^{(n)} \mathcal{D}_{n+j} \quad \text { and } \quad \sum_{i=0}^{\bar{d}} \delta_{i}^{(n)} \mathcal{D}_{n+i}=\mathcal{D}_{n}+\sum_{j=0}^{\bar{d}-1} \mu_{j}^{(n)} \mathcal{W}_{n+j}.
	\end{equation}
	Then
	\begin{equation}\label{Eq50}
		\left\{\begin{array}{l}
			\delta^{(n)}=\hspace{-0.3cm}\underset{\substack{x=\left(x_{0}, \ldots, x_{\bar{d}}\right) \\
					\sum x_{i}=1}}{\operatorname{argmin}}\left\|\mathbb{D}_{\bar{d}}^{n}  \bar{\times}_{(N+1)} x\right\|_{F} \\
			\mathcal{T}_{n}^{(\bar{d})}=\sum_{i=0}^{\bar{d}} \delta_{i}^{(n)} \mathcal{X}_{n+i}
		\end{array}\right.
		\hspace{-0.6cm}\Longleftrightarrow 
		\left\{\begin{array}{l}
			\mu^{(n)}=\hspace{-0.3cm}\underset{\substack{x=\left(x_{0}, \ldots, x_{\bar{d}-1}\right)}}{\operatorname{argmin}}\left\|\mathbb{W}_{\bar{d}-1}^{n} \bar{\times}_{(N+1)} x+\mathcal{D}_{n}\right\|_{F} \\
			\mathcal{T}_{n}^{(\bar{d})}=\mathcal{X}_{n}+\sum_{j=0}^{\bar{d}-1} \mu_{j}^{(n)} \mathcal{D}_{n+j}.
		\end{array}\right.
	\end{equation}
\end{proof} 

\noindent Notice that the $\delta_{i}^{(n)}$'s could  also be computed from the $\mu_{i}^{(n)}$ as follows
\begin{equation}\label{Eq51}
	\delta_{0}^{(n)}=1-\mu_{0}^{(n)}, \quad \delta_{i}^{(n)}=\mu_{i-1}^{(n)}-\mu_{i}^{(n)}, \quad i=1, \ldots, \bar{d}-1, \quad \delta_{\bar{d}}^{(n)}=\mu_{\bar{d}-1}^{(n)}.
\end{equation}

\noindent The steps of these two approaches are  summarized as:
\begin{enumerate}
	
	\item Choose $n$ and $\bar{d}$ and terms $\mathcal{X}_{n}, \mathcal{X}_{n+1}, \ldots, \mathcal{X}_{n+\bar{d}+1}.$
	
	\item Compute the tensors $\mathcal{D}_{n}, \mathcal{D}_{n+1}, \ldots, \mathcal{D}_{n+\bar{d}}$ and form the tensor $\mathbb{D}_{\bar{d}-1}^{k}$.
	
	\item Compute the tensors $\mathcal{W}_{n}, \mathcal{W}_{n+1}, \ldots, \mathcal{W}_{n+\bar{d}-1}$ and form the tensor $\mathbb{W}_{\bar{d}-1}^{n}$.
	
	\item Solve for $x=\left(x_{0}, \ldots, x_{\bar{d}-1}\right)$ the problem
	\begin{equation}\label{Eq52}
		\min _{x}\left\|\mathbb{A}\bar{\times}_{(N+1)} x+\mathcal{B}\right\|_{F},
	\end{equation}
	with
	\begin{equation}\label{Eq53}
		(\mathbb{A}, \mathcal{B})= \begin{cases}\left(\mathbb{D}_{\bar{d}-1}^{n}, \mathcal{D}_{n+\bar{d}}\right) & \text { (approach } 1) \\ \left(\mathbb{W}_{\bar{d}-1}^{n}, \mathcal{D}_{n}\right) & (\text { approach } 2).\end{cases}
	\end{equation}
	
	\item With $x=\left(x_{0}, \ldots, x_{\bar{d}-1}\right)$ available, compute $\delta^{(n)}=\left(\delta_{0}^{(n)}, \delta_{1}^{(n)}, \ldots, \delta_{\bar{d}}^{(n)}\right)$ as
	
	\begin{itemize}
		
		\item Approach 1.
		\begin{equation}\label{Eq54}
			\delta_{i}^{(n)}=\frac{x_{i}}{\sum_{j=0}^{\bar{d}} x_{j}}, i=0,1, \ldots, \bar{d}\quad \text { with } x_{\bar{d}}=1,
		\end{equation}
		\item Approach 2.
		\begin{equation}\label{Eq55}
			\delta_{0}^{(n)}=1-x_{0}, \quad \delta_{i}^{(n)}=x_{i-1}-x_{i}, \quad i=1, \ldots, \bar{d}-1, \quad \delta_{\bar{d}}^{(n)}=x_{\bar{d}-1}.
		\end{equation}
	\end{itemize}
	\item Compute $\mathcal{T}_{n}^{(\bar{d})}$ by 
	
	\begin{equation}\label{Eq56}
		\mathcal{T}_{n}^{(\bar{d})}=\sum_{i=0}^{\bar{d}} \delta_{i}^{(n)} \mathcal{X}_{n+i}.
	\end{equation}
\end{enumerate}
The resulting method following the approach $1$ is TG-MPE while  the approach $2$ corresponds  to   the TG-RRE method.

\section{Implementation via tensor global-QR decomposition}

The purpose of this section is to give an efficient implementation of the two approaches
using the global-QR decomposition given in \cite{17}. 
Let $\mathbb{A}=$ $\left[\mathcal{A}_{1}, \mathcal{A}_{2}, \ldots, \mathcal{A}_{m}\right] \in \mathbb{R}^{I_{1}  \times I_{3} \times \cdots \times I_{N} \times m},$ 
be an $(N+1)$-mode tensor with column tensors 
$\mathcal{A}_{1}, \mathcal{A}_{2}, \ldots, \mathcal{A}_{m} \in \mathbb{R}^{I_{1}  \times I_{2} \times \cdots \times I_{N}}$.
Then, there is an $(N+1)$-mode orthogonal tensor
$\mathbb{Q}=\left[\mathcal{Q}_{1}, \mathcal{Q}_{2},\ldots , \mathcal{Q}_{m} \right] \in
\mathbb{R}^{I_{1} \times I_{2} \times \cdots \times I_{N} \times m},$ 
satisfying 
$\mathbb{Q} \boxdot \mathbb{Q}= I_{m \times m}$
and an upper triangular matrix $R \in \mathbb{R}^{m \times m}$ such that

\begin{equation}\label{Eq58}
	\mathbb{A}=\mathbb{Q} \times_{(N+1)} R^{T}.
\end{equation}
The steps of this decomposition are summarized in the following algorithm.

\begin{algorithm}[H]\label{NNN}.
	\caption{Global tensor QR}
	\begin{algorithmic}[1]
		\item[Inputs:] $\mathbb{A}=\left[ \mathcal{A}_{0}, \mathcal{A}_{2},\ldots,\mathcal{A}_{m-1}\right].$
		\item[Ouputs:]  $R=\left[r_{ij} \right] $ and  $\mathbb{Q}=\left[ \mathcal{Q}_{0}, \mathcal{Q}_{2},\ldots,\mathcal{Q}_{m-1}\right]. $
		\begin{enumerate}
			\item Compute the scalar $r_{0,0}=\langle\mathcal{A}_{0},\mathcal{A}_{0}\rangle^{\frac{1}{2}}$ and $\mathcal{Q}_{0}=\displaystyle \frac{1}{r_{0,0}}\mathcal{A}_{0}.$
			\item For{ $i=1,...,m$}
			\begin{enumerate}
				\item  $\mathcal{U}=\mathcal{A}_{i}.$
				\item For{$j=1,...,i-1$}
				\begin{itemize}
					\item  $r_{j,i}=\langle\mathcal{Q}_{j},\mathcal{U}\rangle,\hspace{0.3cm}
					\mathcal{V}=   \mathcal{U}-r_{j,i}\mathcal{Q}_{j}.$
				\end{itemize} 
				\item EndFor 
				\item $r_{i,i}=\langle\mathcal{U},\mathcal{U}\rangle^{\frac{1}{2}}$. 
				\item  $\mathcal{Q}_{i}=\frac{1}{r_{i,i}}\mathcal{U}$.
			\end{enumerate}
			\item EndFor
		\end{enumerate}
	\end{algorithmic}
\end{algorithm}

\noindent The following result provides an equivalent problem that generalizes for tensors the normal equation. We use it to solve the problem (\ref{Eq52}).
\medskip
\begin{theorem}
	(see \cite{2} ) Let $\mathbb{A}=\left[\mathcal{A}_{1}, \mathcal{A}_{2}, \ldots, \mathcal{A}_{m}\right] \in \mathbb{R}^{I_{1} \times I_{2}  \times \cdots \times I_{N} \times m}$ and $\mathcal{B} \in$ $\mathbb{R}^{I_{1} \times I_{2}  \times \cdots \times I_{N}}$. Then 
	
	\begin{equation}\label{Eq59}
		\tilde{x}=\underset{x \in \mathbb{R}^{n}}{\operatorname{argmin}}\left\|\mathbb{A} \bar{\times}_{(N+1)} x-\mathcal{B}\right\|_{F} \Longleftrightarrow\left(\mathbb{A} \boxdot^{(N+1)} \mathbb{A}\right) \tilde{x}=\mathbb{A} \boxdot^{(N+1)} \mathcal{B},
	\end{equation}
	where
	
	\begin{equation}\label{Eq60}
		\left(\mathbb{A} \boxdot^{(N+1)} \mathbb{A}\right)=\left[\left\langle\mathcal{A}_{i}, \mathcal{A}_{j}\right\rangle\right]_{1 \leq i, j \leq m} \in \mathbb{R}^{m \times m},
	\end{equation}
	and 
	\begin{equation}\label{Eq61} \linebreak \mathbb{A} \boxdot^{(N+1)} \mathcal{B}     =\left[\langle\mathcal{A}_{1},\mathcal{B}\rangle,\langle\mathcal{A}_{2},\mathcal{B}\rangle,\ldots,\langle\mathcal{A}_{m},\mathcal{B}\rangle \right]^{T}\in \mathbb{R}^{m } .\end{equation} 
\end{theorem}

\begin{proposition}\label{p9}
	Let $\mathbb{A}= \mathbb{Q}\times_{(N+1)}R^{T}$ be a  tensor global-QR decomposition of $\mathbb{A}$. Then,
	\begin{equation}\label{Eq62}\mathbb{A}\boxdot^{(N+1)} \mathbb{A}=R^{T}R.\end{equation}
\end{proposition}
\noindent From (\ref{Eq59}) and (\ref{Eq62}) the tensor equation \ref{Eq52} can be transformed into an equivalent  matrix  problem as shown in the following theorem.

\noindent 
\begin{theorem}\label{Thm10}
	With the notations above, we have the following equivalence
	\begin{equation}\label{Eq63}
		\tilde{x}=\underset{x\in \mathbb{R}^{n}}{argmin}\parallel\mathbb{A}\bar{\times}_{(N+1)}x-\mathcal{B} \parallel_{F}\Longleftrightarrow R^{T}R\tilde{x}=b,
	\end{equation}
	where 
	\begin{equation}\label{Eq64}
		b=\mathbb{A} \boxdot^{(N+1)} \mathcal{B}.
	\end{equation}
\end{theorem}

\subsection{ A unified algorithm for TG-MPE and TG-RRE}
Using the global QR-decomposition of $\mathbb{D}^{n}_{\bar{d}-1}$ and $\mathbb{W}^{n}_{\bar{d}-1}$ and taking advantage of Proposition \ref{p9} and Theorem \ref{Thm10}, the implementation steps of TG-MPE and TG-RRE are presented  in the following algorithm.
\begin{algorithm} [H]
	\caption{Implementation of TG-MPE and TG-RRE}\label{TMPE}
	\begin{algorithmic}[1]
		\item[Inputs:] $n$, $\bar{d}$ and $\mathcal{X}_{n},\mathcal{X}_{n+1},\ldots,\mathcal{X}_{n+\bar{d}+1}$.
		\item[Outputs:] $\mathcal{T}_{n}^{\bar(d)}.$
		\begin{enumerate}
			\item  Compute the tensors $\mathcal{D}_{n},\mathcal{D}_{n+1},\ldots,\mathcal{D}_{n+\bar{d}}$ and form the tensor $\mathbb{D}^{n}_{\bar{d}-1}=\left[ \mathcal{D}_{n},\mathcal{D}_{n+1},\ldots,\mathcal{D}_{n+\bar{d}-1}\right].$
			\item  Compute the tensors $\mathcal{W}_{n},\mathcal{W}_{n+1},\ldots,\mathcal{W}_{n+\bar{d}}$ and form the tensor $\mathbb{W}^{n}_{\bar{d}-1}=\left[ \mathcal{W}_{n},\mathcal{W}_{n+1},\ldots,\mathcal{W}_{n+\bar{d}-1}\right] .$
			\item
			\begin{itemize}
				\item 
				TG-MPE,\begin{itemize}
					\item[3.1] Global-QR of $\mathbb{D}^{n}_{\bar{d}-1}$: $\mathbb{D}^{n}_{\bar{d}-1}=\mathbb{Q}\times_{(N+1)}R^{T}$, set $b=\mathbb{D}^{n}_{\bar{d}-1}\boxdot \mathcal{D}_{n+\bar{d}}$.
					\item[3.2] Solve  for $x=(x_{0},x_{1},...,x_{\bar{d} -1})$ the  linear system $R^{T}R\tilde{\theta^{k}}=-b.$\\
					\item[3.3] Compute $\delta^{n}=(\delta^{n}_0,\delta^{n}_1,\ldots,\delta^{n}_{\bar{d}})$ as in (\ref{Eq54}).
				\end{itemize}
				\item TG-RRE,\begin{itemize}
					\item[3.1] Global-QR of $\mathbb{W}^{n}_{\bar{d}-1}$ : $\mathbb{W}^{n}_{\bar{d}-1}=\mathbb{Q}\times_{(N+1)}R^{T}$, set $b=\mathbb{D}^{n}_{\bar{d}-1}\boxdot \mathcal{W}_{n}$.\\
					\item[3.2] Solve  for $x=(x_{0},x_{1},...,x_{\bar{d} -1})$ the  linear system $R^{T}R\tilde{\theta^{k}}=-b.$\\
					\item[3.3] Compute $\delta^{n}=(\delta^{n}_0,\delta^{n}_1,\ldots,\delta^{n}_{\bar{d}})$ as in (\ref{Eq55}).
				\end{itemize}
			\end{itemize}
			\item Compute the approximation  
			$\mathcal{T}_{n}^{\bar{d}}= \displaystyle \sum_{i=0}^{\bar{d}-1}\delta^{n}_{i} \mathcal{X}_{i}.$
		\end{enumerate}
	\end{algorithmic}
\end{algorithm}

\section{TG-MPE and TG-RRE as nonlinear Krylov subspace methods}
Recall that the $k^{th}$ tensor Krylov subspace $\mathcal{K}_{k}(\mathcal{M},\mathcal{V})$ associated to the pair of tensors $(\mathcal{M}, \mathcal{V})$   is defined by
$$\mathcal{K}_{k}(\mathcal{M},\mathcal{V})=span\left\lbrace \mathcal{V}, \mathcal{M} *_{N} \mathcal{V},\mathcal{M}^{2} *_{N} \mathcal{V} \ldots, \mathcal{M}^{(k-1)} *_{N} \mathcal{V}\right\rbrace.$$
As demostrated in \cite{17}, TG-MPE and TG-RRE, when applied to the generated linear sequence (\ref{Eq21}),
\begin{equation}\label{Seq}
	\mathcal{X}_{n+1}=\mathcal{G}(\mathcal{X}_{n})=\mathcal{M} *_{N} \mathcal{X}_{n}+\mathcal{B},
\end{equation}  
\noindent are mathematically equivalent to some well known Krylov subspace methods such as the tensor GMRES. Let $\mathbb{D}_{(n,d)},$ and $\mathbb{W}_{(n,d)}$ be  the spaces generated by differences $\left\lbrace \mathcal{D}_{n+k}\right\rbrace _{k=0}^{d}$ and $\left\lbrace \mathcal{W}_{n+k}\right\rbrace_{k=0}^{d}$, respectively where  $\mathcal{D}_{n+k}$ and $\mathcal{W}_{n+k}$,  are as in (\ref{Eq23}),
$$\mathbb{D}_{(n,d)}=span\left\lbrace \mathcal{D}_{n}, \mathcal{D}_{n+1}, \ldots, \mathcal{D}_{n+d}\right\rbrace \quad and 
\quad \mathbb{W}_{(n,d)}= span \left\lbrace\mathcal{W}_{n}, \mathcal{W}_{n+1}, \ldots, \mathcal{W}_{n+d}\right\rbrace.$$
\noindent The produced approximations $\mathcal{T}_{n}^{\bar{d}}$ in (\ref{Eq56}) are such that  the associated (generalized) residual $$\bar{\mathcal{R}}_{n} =\mathcal{T}_{n+1}^{\bar{d}}-\mathcal{T}_{n}^{(\bar{d})}=\mathcal{B}-(\mathcal{I}-\mathcal{M}) *_{N} \mathcal{T}_{n}^{(\bar{d})},$$ satisfies the relations
$$\begin{cases}\mathcal{T}_{n}^{(\bar{d})}-\mathcal{X}_{n}=\sum_{j=0}^{\bar{d}-1} \mu_{j}^{(n)} \mathcal{D}_{n+j}\in \mathbb{D}_{(n,\bar{d}-1)}  \\ \bar{\mathcal{R}}_{n} \perp \mathbb{V} \end{cases}
with \quad \mathbb{V}= \begin{cases}\mathbb{D}_{(n,\bar{d}-1)} & \text { for TG-MPE } \\ \mathbb{W}_{(n,\bar{d}-1)} & \text {for TG-RRE }. \end{cases}$$

\noindent  Following (\ref{Eq28}), $\mathbb{D}_{(n,\bar{d}-1)}$ and  $\mathbb{W}_{(n,\bar{d}-1)}$ have  Krylov structures given by 
$$\begin{cases} \mathbb{D}_{(n,\bar{d}-1)}=span\left\lbrace \mathcal{D}_{n}, \mathcal{M} *_{N} \mathcal{D}_{n},\mathcal{M}^{2} *_{N} \mathcal{D}_{n} \ldots, \mathcal{M}^{\bar{d}-1} *_{N} \mathcal{D}_{n}\right\rbrace=\mathcal{K}_{\bar{d}}(\mathcal{M},\mathcal{D}_{n}),\\
	\mathbb{W}_{(n,\bar{d}-1)}=\mathcal{M}*_{N}\mathbb{D}_{(n,\bar{d}-1)}=\mathcal{M}*_{N}\mathcal{K}_{\bar{d}}(\mathcal{M},\mathcal{D}_{n}).
\end{cases}$$

\noindent Taking into account this connection, one can investigate an equivalent implementation  based on the tensor orthogonalization process of Arnoldi. Algorithm \ref{Ar} highlights the steps of this implementation.

\begin{algorithm}.
	\caption{Arnoldi based Implementation of TG-MPE and TG-RRE}\label{Ar}
	\begin{algorithmic}[h]
		\item[Inputs:] $n$, $\bar{d}$ and $\mathcal{X}_{n},\mathcal{X}_{n+1},\ldots,\mathcal{X}_{n+\bar{d}+1}$.
		\item[Outputs:]  $\mathcal{T}_{n}^{\bar(d)}.$\begin{enumerate}

			\item  Compute the tensors $\mathcal{D}_{n},\mathcal{D}_{n+1},\ldots,\mathcal{D}_{n+\bar{d}}$.
			
			\item  Compute  $\beta = \parallel \mathcal{D}_{n}\parallel_{F}$ and $\mathcal{V}_{n}=\mathcal{D}_{n}/\beta.$
			
			\item For {$k=1,...,\bar{d}$}\begin{itemize}

				\item Set $\mathcal{Y}=\mathcal{D}_{n+k}. $ 
				
				\item Orthogonalize $\mathcal{Y}$ with respect to $\left\lbrace \mathcal{V}_{n},\mathcal{V}_{n+1},...,\mathcal{V}_{n+k-1} \right\rbrace$ (=$\mathcal{K}_{k}(\mathcal{M},\mathcal{V}_{n})$ in linear case).
				
				\item For {$j=1,...,k$} 
				\begin{itemize}
					\item  $h_{j,k}=\langle\mathcal{V}_{j},\mathcal{Y}\rangle.$
					\item $\mathcal{Y}=  \mathcal{Y}-h_{j,k}\mathcal{V}_{j}.$
				\end{itemize}
				
				\item EndFor
				
				\item $h_{k+1,k}=\langle\mathcal{Y},\mathcal{Y}\rangle^{\frac{1}{2}}$, $\mathcal{V}_{n+k}=\mathcal{V}/h_{k+1,k}.$

			\end{itemize}
			\item EndFor
			\item Set  $H_{\bar{d}}=\left[ h_{ij} \right]_{1\leq i,j\leq \bar{d}} \in\mathbb{R}^{\bar{d}\times \bar{d}}$ and $\hat{H}_{\bar{d}}=\left[h_{ij} \right]_{1\leq i\leq (\bar{d}+1), 1 \leq j\leq \bar{d}}\in\mathbb{R}^{(\bar{d}+1)\times \bar{d}}$.
			
			\item Determine $\mu^{(n)}=(\mu^{(n)}_{0},\mu^{(n)}_{1},...,\mu^{(n)}_{\bar{d} -1})$ via
			\begin{itemize}
				\item TG-MPE:
				$\mu^{(n)}=(H_{\bar{d}})^{-1}b\quad with \quad b=\beta e_{1}.$
				
				\item TG-RRE: $\mu^{(n)}= \underset{\mu}{argmin} \parallel(\hat{H}_{\bar{d}})\mu-\beta e_{1}\parallel.$
				
			\end{itemize}
			
			\item Compute $\mathcal{T}_{n}^{(\bar{d})}=\mathcal{X}_{n}+\sum_{j=0}^{\bar{d}-1} \mu^{(n)}_{j} \mathcal{V}_{n+j}.$
		\end{enumerate}
	\end{algorithmic}
\end{algorithm}

\noindent An interesting observation here is that Algorithm \ref{Ar} depends solely on the tensor terms 
$\mathcal{D}_{n}, \mathcal{D}_{n+1}, \ldots, \mathcal{D}_{n+\bar{d}}$ 
and is independent of how the sequence $\lbrace  \mathcal{X}_{n}\rbrace $ is
generated, whether it is linear or nonlinear.  As a direct consequence, tensor extrapolation methods, when applied to nonlinear sequences , can be considered as nonlinear adaptations of tensor Krylov subspace methods since they employ the same orthogonalization processes. Moreover, this allow us to use the advanced techniques employed in tensor Krylov subspace methods  to improve tensor extrapolation methods. It also suggests that any tensor  Krylov type method has a tensor extrapolation method counterpart, and conversely. 

\section{Numerical experiments}
In this section, we present some  numerical examples to show  the effectiveness of TG-MPE and TG-RRE algorithms.  The first example contains two tests comparing TG-MPE and TG-RRE  with  the tensor biconjugate gradient decent BiCR2 presented in \cite{18}. The second example is devoted to some experiments on completion problems. The last one deals with a nonlinear sequence. All computations were done  with  MATLAB2021a on a
PC HP with 4 GHz  and 16 GB RAM. We stop the computations once the relative error $\dfrac{\|\mathcal{X}_k-\mathcal{X}_{exact}\|_{F}}{\|\mathcal{X}_{exact}\|_{F}} $ is less than $10^{-14} $. The maximum iteration number is set to be $400$ if not specified.

\subsection{Example 1} In this example, we apply TG-MPE and TG-RRE on the extended BiCR2 method  presented in \cite{18} for solving the general based Einstein-product equation of  the form 
\begin{equation}\label{eq1}
	\mathcal{A}\ast_M\mathcal{X}\ast_N\mathcal{B}=\mathcal{C}.
\end{equation}
As BiCR2  is a conjugate gradient-like method, its convergence  is in general slow.  Throughout the two tests below, we show how TG-MPE and TG-RRE can improve its convergence rate. For measuring the accuracy of  the underlying methods, i.e, BiCR2, BiCR2+TG-MPE and BiCR2+TG-RRE,  we adopt the  relative error and the relative residual,
$$ Er=\dfrac{\|\mathcal{X}_k-\mathcal{X}_{exact}\|_{F}}{\|\mathcal{X}_{exact}\|_{F}}\quad\text{and}\quad Re=\dfrac{\|\mathcal{A}\ast_M\mathcal{X}_k\ast_N\mathcal{B}-\mathcal{C}\|_{F}}{\|\mathcal{C}\|_{F}},$$
where $\mathcal{X}_{exact}$ is the exact solution.

\subsubsection{Experiment  1} 
Consider  the tensor equation \begin{equation}\label{Eq65}
	\mathcal{A}\ast_2\mathcal{X}=\mathcal{C},
\end{equation}
with $\mathcal{A}  \in \mathbb{R}^{7\times 7 \times 7\times 7}$ is a random tensor 
and the right hand side tensor $\mathcal{C}$ 
is computed via (\ref{Eq65}) by considering
the exact solution $\mathcal{X}_{exact} =  \in \mathbb{R}^{7\times 7 \times 7\times 7}$ with all elements equal to one and the initial guess is set to the zero tensor.

\noindent Figure \ref{Fig1}  illustrates the behaviour of the general error norm  $Er$ and  the residual norm  $Re$ versus the iteration number, while Figure \ref{Fig2}  shows their behaviours versus  the CPU time. The curves  reveal clearly the acceleration effect of TG-MPE and TG-RRE on BiCR2 in terms of the number of iterations number  and the required computational time.

\begin{figure}[h]
	\centering
	
	\includegraphics[width=0.4\textwidth]{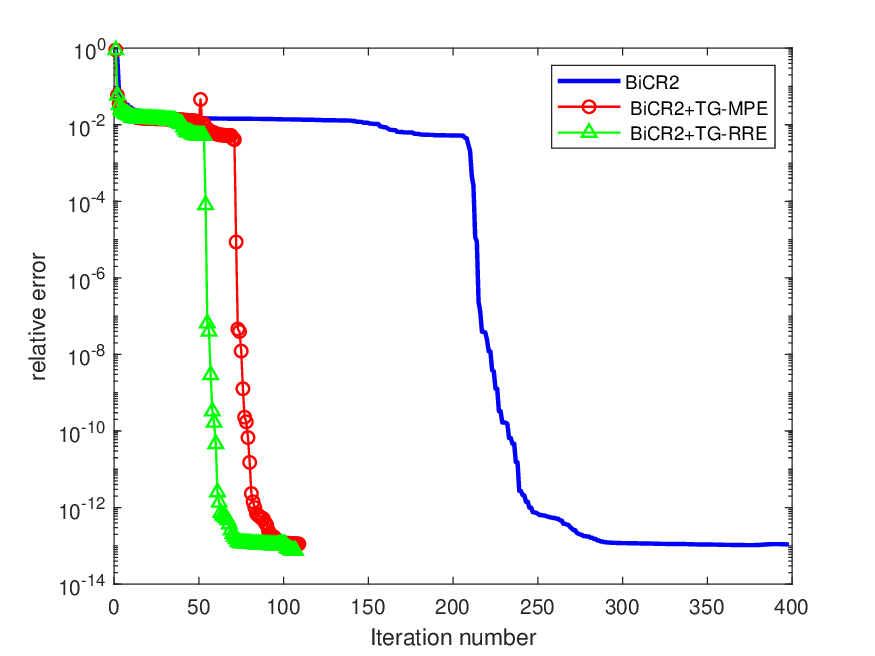}
	\includegraphics[width=0.4\textwidth]{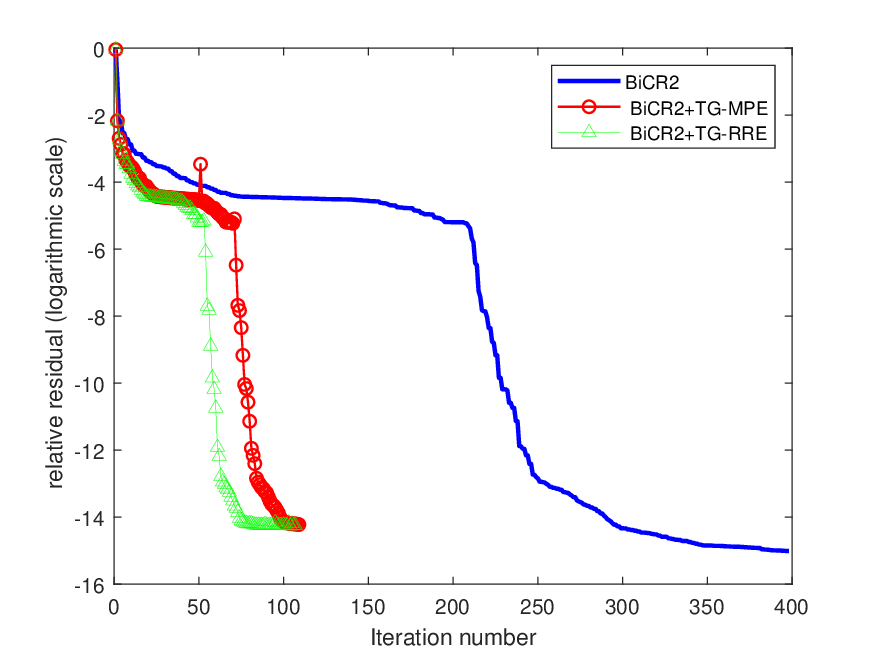}
	
	\caption{Experiment 1: The relative error(left) and the relative residual (right) versus iteration number.}
	\label{Fig1}
\end{figure}
\begin{figure}[h] 
	\centering
	
	\includegraphics[width=0.4\textwidth]{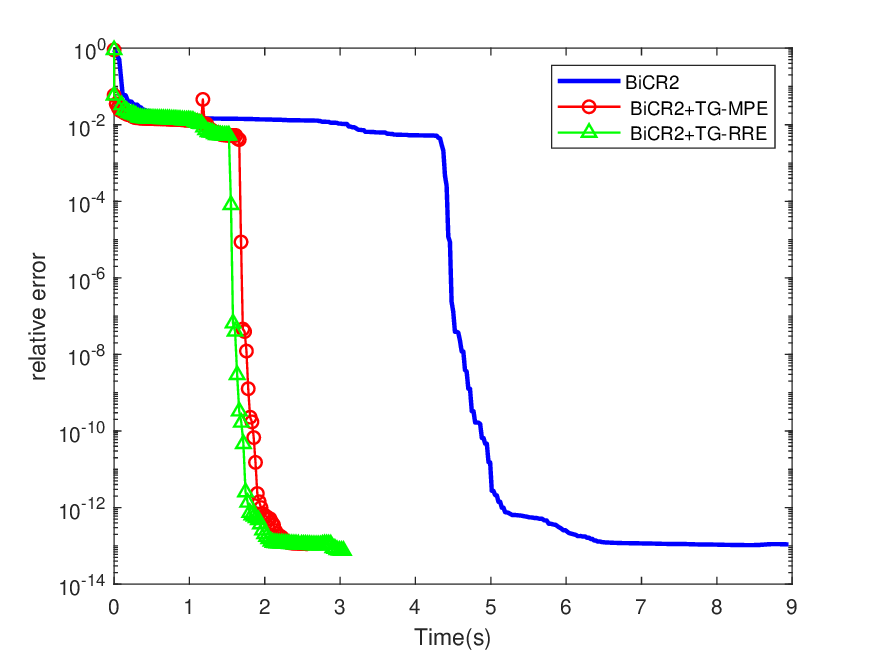}
	\includegraphics[width=0.4\textwidth]{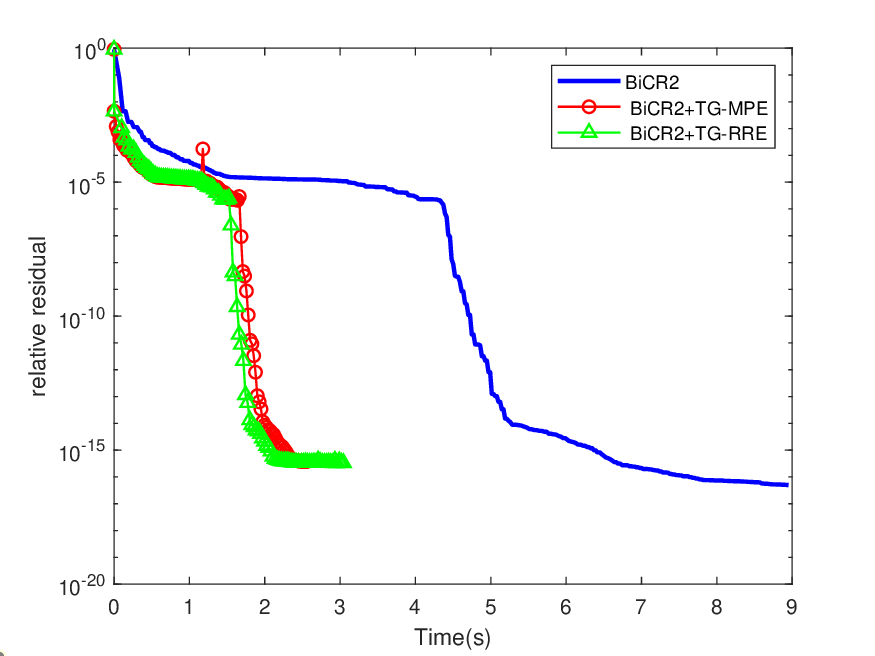}    
	\caption{Experiment 1: The   relative error (left) and the relative residual  (right) versus the required CPU time.}
	\label{Fig2}
\end{figure}

\newpage
\subsubsection{Experiment 2} 
For this second experiment,  we consider the more general tensor equation \begin{equation}\label{Eq66}
	\mathcal{A}\ast_M\mathcal{X}\ast_N\mathcal{B}=\mathcal{C}.
\end{equation}
where  $\mathcal{A}$ and $\mathcal{B}$ are arbitrary random tensors,
$\mathcal{A} = tenrand(7, 7, 6, 5) $ and $\mathcal{B} = tenrand(4, 4, 7, 7)$. 
The right hand side tensor $\mathcal{C}$ 
is computed via (\ref{Eq66}) such that the
exact solution is as $\mathcal{X}_{exact} = tensor(ones(7, 7, 4, 4)).$
We start again by the zeros tensor $\mathcal{X}_0 $. As in experiment $1$,  we plot  error $Er$  and residual   $Re$ versus the number of  iterations (Figure \ref{Fig3}) and CPU time (Figure \ref{Fig4}). 

\begin{figure}[h]
	\centering
	
	\includegraphics[width=0.4\textwidth]{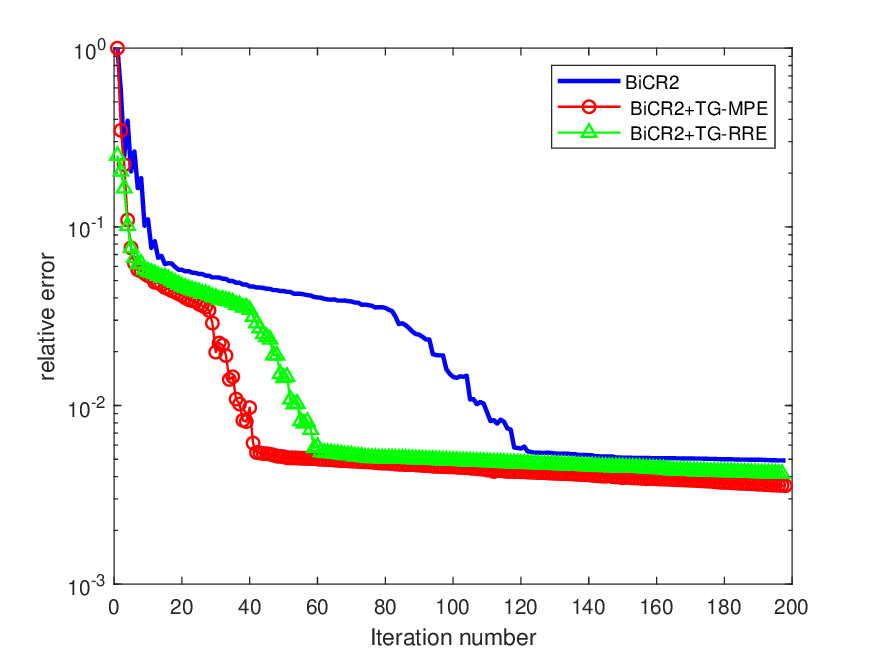}
	\includegraphics[width=0.4\textwidth]{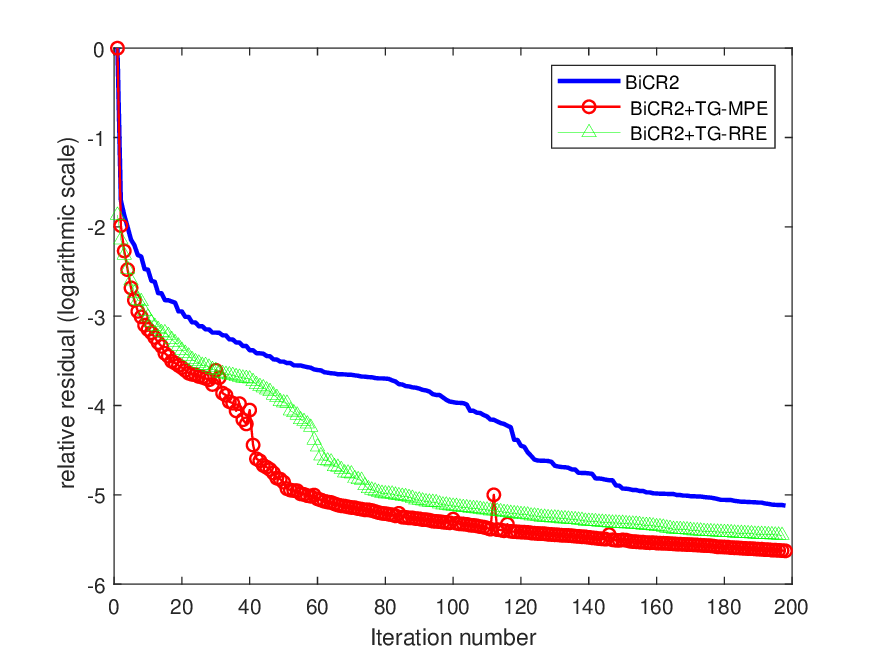}
	
	\caption{Experiment 2: The  relative error (left) and the relative residual (right) versus the  iteration number.}
	\label{Fig3}
\end{figure}
\begin{figure}[h] 
	\centering
	
	\includegraphics[width=0.4\textwidth]{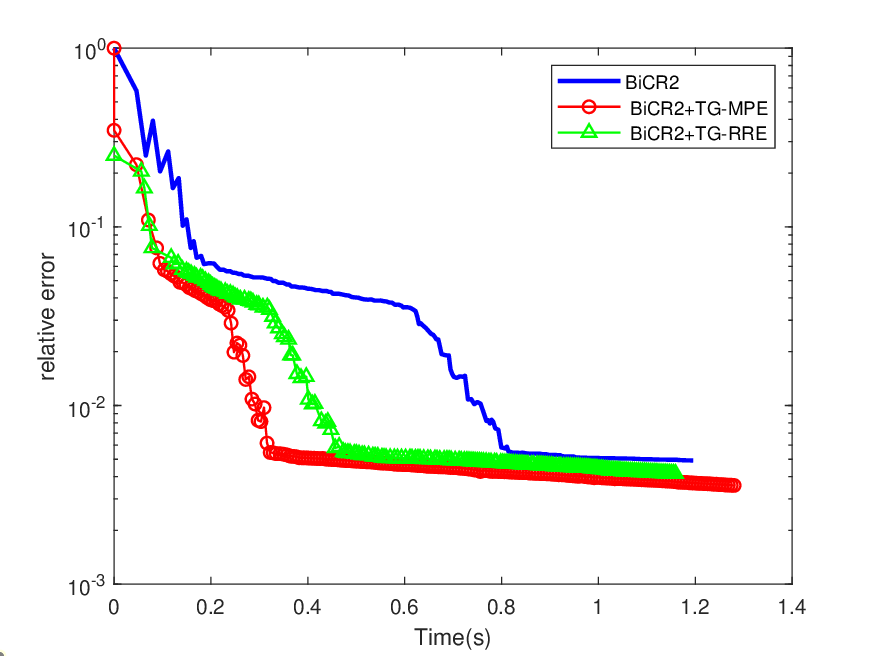}
	\includegraphics[width=0.4\textwidth]{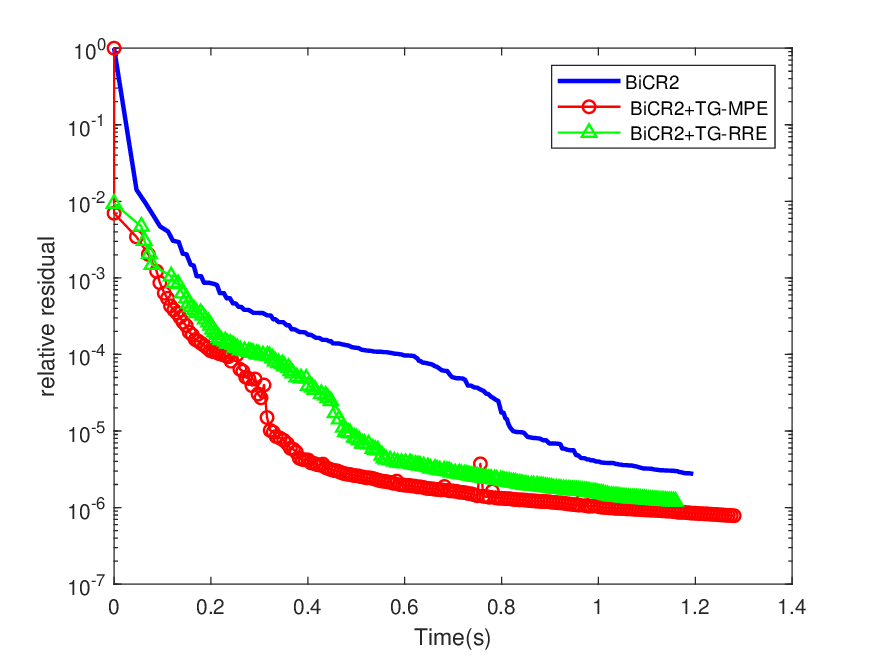}
	
	\caption{Experiment 2: The  relative error (left) and the relative residual (right) versus the required CPU time. }
	\label{Fig4}
\end{figure}
\noindent The obtained results confirm the accuracy of TG-MPE and TG-RRE. Furthermore, the curves reveal that TG-RRE is more stable than TG-MPE.

\subsection{Example 2} This example is devoted to the completion problem  that aims   the reconstruction of a low-rank symmetric tensor from its incomplete (or  randomly corrupted) observed entries (\cite{comp1,comp2,comp4,comp3}). The task is  to estimate a symmetric order-three tensor $\bar{\mathcal{V}} \in \mathbb{R}^{N\times N \times N}$ from an available  noisy and incomplete tensor $\mathcal{V} \in \mathbb{R}^{N\times N \times N} $
whose entries  are as follows
\begin{equation}\label{Eq67} \mathcal{V}(i,j,k)=\left\lbrace
	\begin{aligned}
		&\bar{\mathcal{V}}(i,j,k)+\mathcal{B}(i,j,k), & (i,j,k) \in \Omega , \\
		0	& \hspace{3cm} &otherwise,\\			
	\end{aligned}
	\right.\end{equation}
where $\mathcal{B}(i,j,k)$ is the noise  at location $(i,j,k)$, and $\Omega\subseteq \{1,2,...,N\}^{3}$ is the index subset of the observed entries. The tensor $\bar{\mathcal{V}}$ is such that
\begin{equation}\label{Eq68}
	\bar{\mathcal{V}}=\sum_{k=1}^{r}\bar{v}_{k}^{\otimes 3},
\end{equation} 
where $\bar{v}_{k}^{\otimes 3}=  \bar{v}_{k}\otimes\bar{v}_{k}\otimes\bar{v}_{k}$ with $\otimes$ stands for the outer product and $\bar{v}_1,...,\bar{v}_r \in \mathbb{R}^{N}$.

\noindent For estimating the  unknown vectors $\bar{v}_1,...,\bar{v}_2$, the authors in \cite{comp2} adopt the following minimization problem 
\begin{equation}\label{Eq69}
	\underset{V\in \mathbb{R}^{N\times r}}{min}~\{ g(V)= \parallel Proj_{\Omega}(\sum_{k=1}^{r}v_{k}^{\otimes 3}-\mathcal{V})\parallel_{F}^{2} \},	
\end{equation}
where $V=\left[v_{1},...,v_r \right] \in  \mathbb{R}^{N\times r}$ and $Proj_{\Omega}(\mathcal{T})$ is the orthogonal projection of the tensor $\mathcal{T}$ onto the subspace of tensors   vanishing outside of $\Omega$.
For solving  (\ref{Eq69}) the authors in (\cite{comp1}) proposed a  gradient decent  based algorithm with spectral initialisation (referred here by \textbf{ SpecGD}) which convergence is within nearly linear time.(see \cite{comp1} for more details).
\noindent We generate randomly  data $\bar{\mathcal{V}}$, $\bar{\mathcal{B}}$ and $\Omega$, and compute the corrupted (incomplete) tensor $\mathcal{V}$ via (\ref{Eq67}). In Figures \ref{Fig5},    \ref{Fig6} and \ref{Fig7}  we plotted, for the $first 100$ iterations, the obtained results  for different dimensions $N= 15, 30$ and $50$.\\
The plots  on  the left of each figure show the behaviour of  relative error 
$Er=\frac{\parallel \mathcal{V}_k -\bar{\mathcal{V}}\parallel_{F}}{\parallel\bar{\mathcal{V}}\parallel_{F}}$
For SpecGD, SpecGD+TG-MPE and SpecGD+TG-RRE  while the plots on the    right  illustrate relative error versus the required CPU time (in seconds). The figures show that the convergence rate of  SpecGD is improved when apply on it the extrapolation methods TG-MPE and TG-RRE. For $N=15$, we observe that, for the first $25$ iterations, TG-MPE is faster than TG-MPE. For $N=30$ and $N=50$    we observe that TG-RRE is more accurate. That suggests that TG-RRE is more appropriate when dimensions become larger.

\begin{figure}[H]
	\centering
	
	\includegraphics[width=0.45\textwidth]{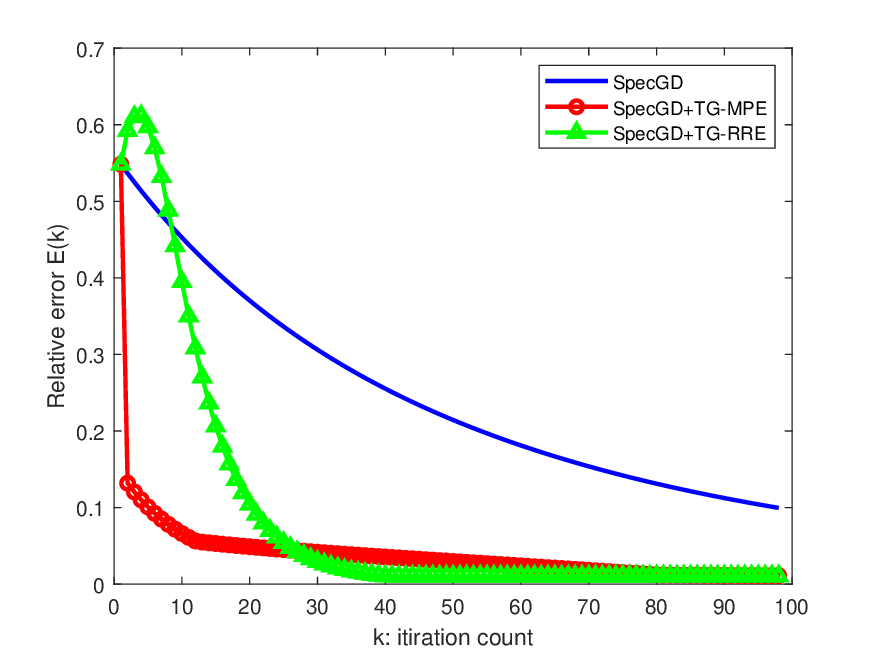}
	\includegraphics[width=0.45\textwidth]{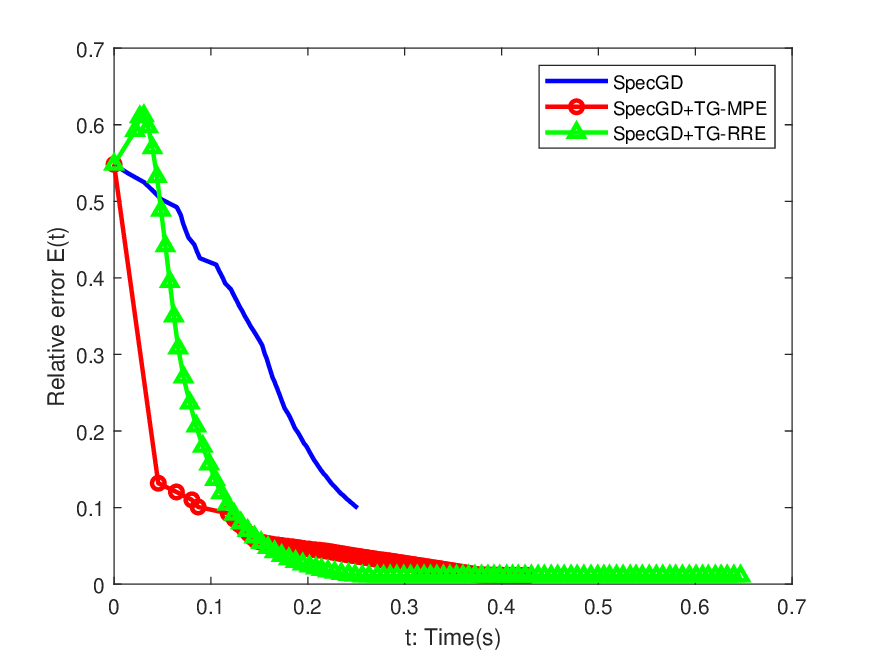}
	
	\caption{The relative error versus iteration number  (left) and CPU time (right) for $N=15$. }
	\label{Fig5}
\end{figure}
\begin{figure}[H] 
	\centering
	\includegraphics[width=0.45\textwidth]{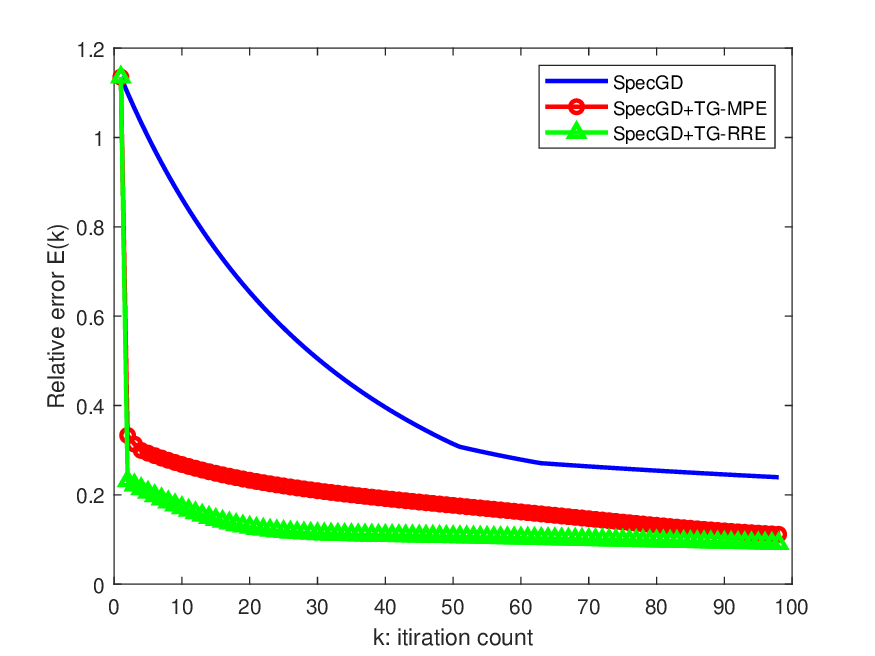}
	\includegraphics[width=0.45\textwidth]{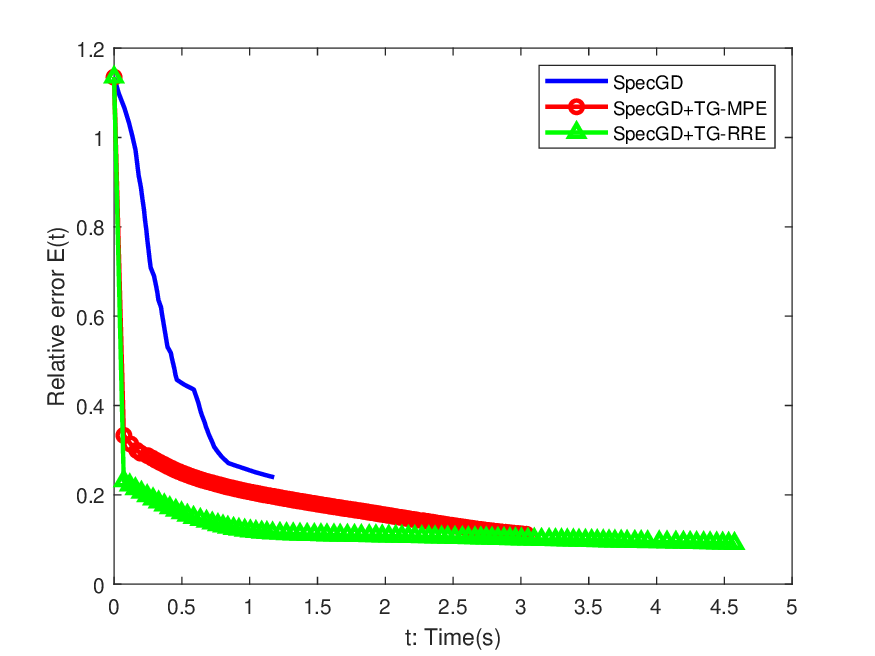}
	
	\caption{The relative error versus iteration number  (left) and CPU time (right) for $N=30$. }
	\label{Fig6}
\end{figure}

\begin{figure}[H] 
	\centering
	\includegraphics[width=0.45\textwidth]{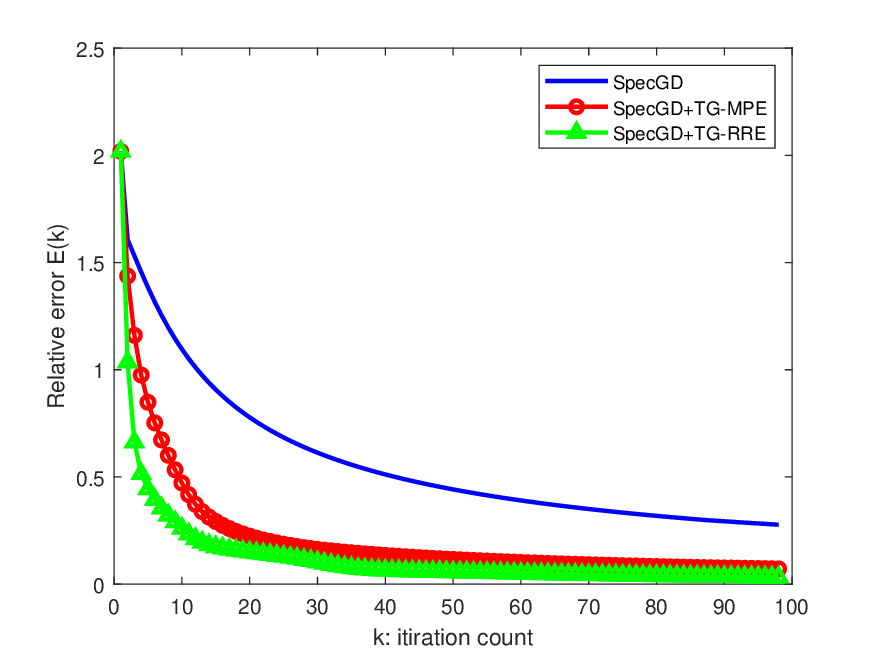}
	\includegraphics[width=0.45\textwidth]{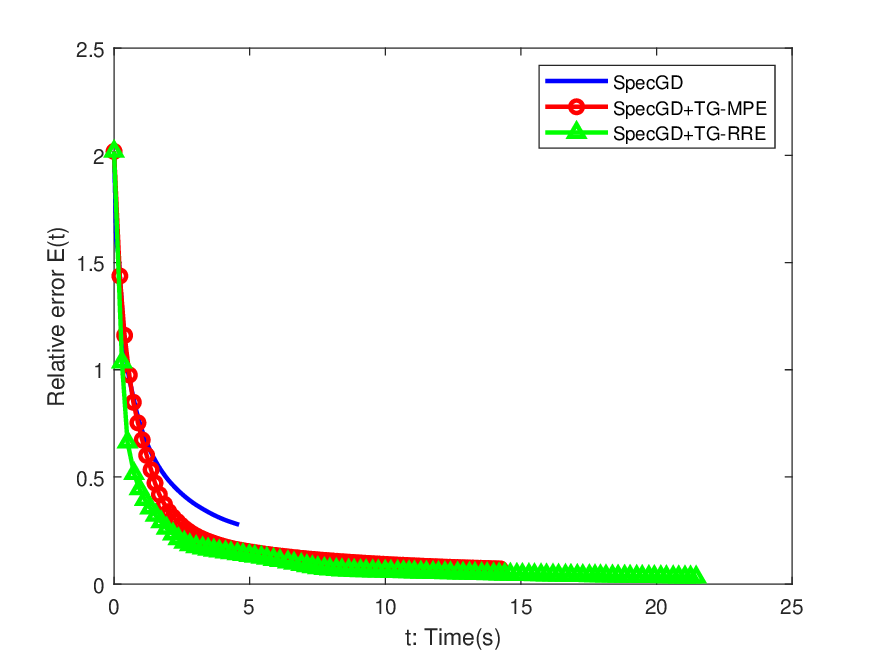}
	
	\caption{The relative error versus iteration number  (left) and CPU time (right) for $N=50$.}
	\label{Fig7}
\end{figure}

	
	\subsection{Example 3}
	We  consider here a nonlinear  sequence $\left\lbrace \mathcal{X}_n \right\rbrace \in \mathbb{R}^{N\times N\times N}$ that is generated by the relation 
	\begin{equation}
		\mathcal{X}_{n+1}= sin\left((1+\frac{1}{(n+1)^2})\mathcal{X}_n)\right), 
	\end{equation}
	\noindent where the initial guess $\mathcal{X}_0$ is a random  nonzero tensor. It is obvious that the limit of this sequence is the zero tensor. We then consider the relative error $Er=\frac{\parallel\mathcal{X}^{method}_n\parallel}{\parallel\mathcal{X}_0\parallel}$ to measure the accuracy of our methods. We stop computations when iteration number $n$ reach $200$.
	
	\noindent Figure \ref{Fignon} shows the relative error  behaviour of the basic iterations (without extrapolation), TG-MPE and TG-RRE for different dimensions $N=5$ (left) and $N=10$ (right).
	\noindent The obtained results suggests that the tensor extrapolation methods, TG-MPE and TG-RRE are  useful tools for accelerating the convergence of some nonlinear sequences. We can also see the superiority of TG-RRE against the TG-MPE.  
	
	\begin{figure}[H] 
		\centering
		\includegraphics[width=0.45\textwidth]{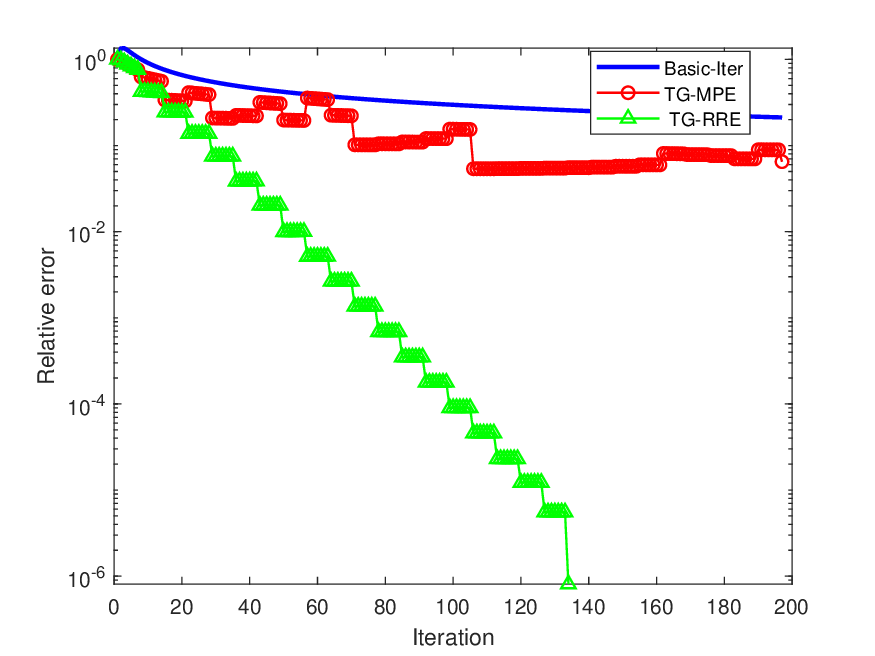}
		\includegraphics[width=0.45\textwidth]{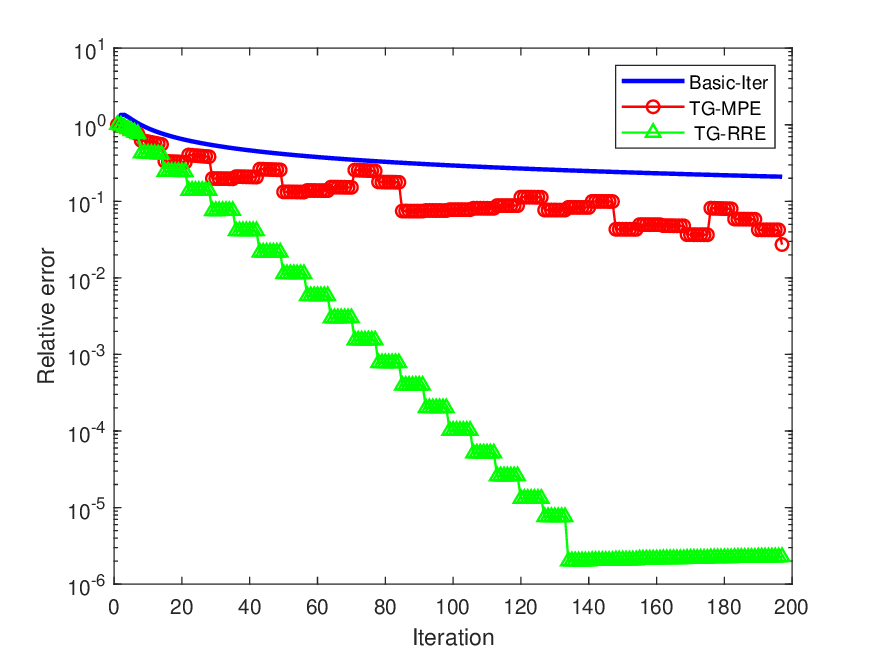}
		\caption{The relative error versus the iteration number for different dimensions $ N=5$  (left) and $N=10$ (right) ( Example 3).}
		\label{Fignon}
	\end{figure}

	\section{Conclusion}
	 In this paper, we establish the tensor extrapolation method TG-MPE and TG-RRE derived throughout the framework of linear tensor equations in the form $\mathcal{A}*_N \mathcal{B} =\mathcal{C}$,  The followed approach allows us to demonstrate  the connection with polynomial methods as well as justify how these methods can be thought nonlinear Krylov subspace methods when applied to nonlinear problems. To confirm the effectiveness of the TG-MPE and TG-RRE, we have applied them to  some linear and nonlinear generated sequences. The obtained results confirm the feasibility and performance of the methods.
\section*{Acknowledgements:}

\noindent The work was partially supported by the Moroccan Ministry of Higher Education, Scientific Research and Innovation and the OCP Foundation through the APRD research program.

\bibliography{mybibfile}

\end{document}